\documentclass[11pt,reqno]{amsart}

\usepackage{amscd,amsmath,amsopn,amssymb,amsthm,multicol,xcolor}
\usepackage{graphicx,mathrsfs}
\usepackage{ifthen}
\usepackage{tikz}\usetikzlibrary{matrix,arrows,decorations.pathmorphing,cd}
\usepackage[all]{xypic}

\definecolor{carmine}{rgb}{0.59, 0.0, 0.09}
\definecolor{mediumpersianblue}{rgb}{0.0, 0.4, 0.65}
\definecolor{persianplum}{rgb}{0.44, 0.11, 0.11}

\usepackage[colorlinks=true,
            linkcolor=persianplum, 
            urlcolor=olive,
            citecolor=mediumpersianblue, 
            backref=page]{hyperref}

\newcommand\bond[1]{\draw (#1) -- +(1,0)}
\newcommand\tcirc[3]{
	\ifthenelse{\equal{#1}{w}}{\filldraw[fill=white,draw=black] (#2) circle (0.08);}{}%
	\ifthenelse{\equal{#1}{b}}{\filldraw[black] (#2) circle (0.08);}{}%
	\draw (#2) node[above=2pt] {#3};
	}
\newcommand\tcross[2]{
	\draw (#1) node[above=2pt] {#2};
	\draw (#1) ++(-0.12,-0.12)-- +(0.24, 0.24);
	\draw (#1) ++(-0.12, 0.12)-- +(0.24,-0.24);
	}
\newcommand\DDnode[3]{
\ifthenelse{\equal{#1}{w}}{\tcirc{w}{#2}{#3}}{}		
\ifthenelse{\equal{#1}{b}}{\tcirc{b}{#2}{#3}}{}		
\ifthenelse{\equal{#1}{x}}{\tcross{#2}{#3}}{}		
}

\renewcommand\1{{\bf 1}}
\renewcommand\a{\alpha}
\renewcommand\b{\beta}

\newcommand\com[1]{}
\newcommand\CC{{\mathbb C}}

\newcommand\E{\mathcal{E}}
\newcommand\fp{\mathfrak{p}}
\newcommand\g{{\mathfrak g}}
\newcommand\gc[1]{\color{gray}#1}
\newcommand\h{\mathfrak{h}}
\newcommand\I{{\mathcal I}}
\renewcommand\l{\lambda}

\newcommand\m{\mathfrak{m}}

\newcommand\op[1]{\mathop{\rm #1}\nolimits}

\newcommand\p{\partial}
\newcommand\PP{{\mathbb P}}
\renewcommand\qed{\phantom{\underline{y}}\hfill\hfill$\square$}
\newcommand\RR{{\mathbb R}}

\newcommand\V{{\mathcal V}}
\newcommand\W{{\mathcal W}}

\newcommand\z{\sigma}

\theoremstyle{plain} 
\newtheorem{theorem}{Theorem}[section]
\newtheorem{proposition}[theorem]{Proposition}
\newtheorem{corollary}[theorem]{Corollary}
\newtheorem{lemma}[theorem]{Lemma}
\theoremstyle{definition}
\newtheorem{definition}[theorem]{Definition}

\newtheorem{remark}[theorem]{Remark}

\newcommand\Cc{{\mathscr C}}

\newcommand{\cT}{\mathcal{T}}

\newcommand{\w}{{\,{\wedge}\;}}
\newcommand{\ve}{\varepsilon}

 \newcommand{\cG}{\mathcal{G}}
 
 \newcommand{\cJ}{\mathcal{J}}
 \newcommand{\cK}{\mathcal{K}}

\newcommand{\half}{\textstyle{\frac 12}}

\newcommand{\txi}{\tilde\xi}
\newcommand{\tomega}{\tilde\omega}
\newcommand{\tphi}{\tilde\phi}
\newcommand{\tpsi}{\tilde\psi}
\newcommand{\tXi}{\tilde\Xi}
\newcommand{\tOmega}{\tilde\Omega}
\newcommand{\tPhi}{\tilde\Phi}

\newcommand{\tL}{\tilde L}
\newcommand{\tJ}{\tilde J}
\newcommand{\tg}{\tilde g}
\newcommand{\ttau}{\tilde \tau}

\begin{document}

\title[Zero-curvature and integrability in 5D]{Zero-curvature subconformal structures\\
and dispersionless integrability in dimension five}

\author[Boris Kruglikov]{Boris Kruglikov$^\dagger$}
\address{$^\dagger$ Department of Mathematics and Statistics, UiT the Arctic University of Norway, Troms\o\ 9037, Norway.}

\author[Omid Makhmali]{Omid Makhmali$^\dagger$$^\ddagger$}
\address{$^\ddagger$ Departamento de Geometr\'ia y Topolog\'ia and IMAG, Universidad de Granada, Granada 18071, Spain.}

\address{Email addresses:\qquad {\tt boris.kruglikov@uit.no}\quad\text{\rm and }\quad {\tt omid.makhmali@uit.no,\quad omakhmali@ugr.es}\hspace{1pt}.}


 \begin{abstract}
We extend the recent paradigm ``Integrability via Geometry'' from dimensions 3 and 4 to higher dimensions, 
relating dispersionless integrability of partial differential equations to curvature constraints of the background geometry.
We observe that in higher dimensions on any solution manifold the symbol defines a vector distribution equipped with 
a subconformal structure, and the integrability imposes a certain compatibility between them.

In dimension 5 we express dispersionless integrability via the vanishing of a certain curvature of
this subconformal structure.
We also obtain a ``master equation'' governing all second order dispersionless integrable equations in 5D,
and count their functional dimension.

It turns out that the obtained background geometry is parabolic of the type $(A_3,P_{13})$. 
We provide its Cartan theoretic description and compute the harmonic curvature components via the Kostant theorem.
Then we relate it to 3D projective and 4D conformal geometries via twistor theory, discuss symmetry reductions and 
nested Lax sequences, as well as give another interpretation 
of dispersionless integrability in 5D through Levi-degenerate CR structures~in~7D.
 \end{abstract}

\maketitle
\tableofcontents

%
%

\section{Introduction and main results}\label{Sec:1}

Lax pair formulation is one of the  most conventional approaches in the integrability analysis of differential equations.
For solitonic type equations the integrability is related to zero-curvature representation \cite{FT,KW}, while for 
dispersionless PDEs the situation is different. Namely, dispersionless Lax pair (dLp) no longer consists 
of higher order differential operators, but is reduced to a pair of vector fields with a spectral parameter, 
often also with terms including differentiation by the spectral parameter \cite{Za}. 
Such dLp allows convenient geometrization of the integrability via twistor theory \cite{JT,Hi,MW}, which eventually
for many classes of dispersionless systems in a more general context has been related to curvature properties 
of underlying background geometries \cite{MD,C,DFKN1,DFKN2}.



The integrability of second order PDEs, or more generally systems of PDEs
with quad\-ratic characteristic variety, was shown to be equivalent to the Einstein-Weyl property in 3D and
the self-duality property in 4D for the canonical conformal structure read off  the characteristic variety/symbol
of the equation \cite{CK}. Here the integrability is understood as the existence of a nontrivial dLp. 
Moreover, for several classes of examples, this was also shown to be equivalent to the existence of
hydrodynamic reductions, as developed in \cite{Ts,FKh}, see \cite{FK}.

Not much is known in dimensions beyond 4, except for sporadic examples \cite{Ta,PP,B,KS}.
In this paper we address the smallest of those dimensions, namely 5D. Let us note that
conformal geometry in 4D is one of the key examples of the so-called parabolic geometry, and
that self-duality is expressed as vanishing of a part of its curvature, namely the anti-self-dual Weyl tensor. Similarly,
in 3D, Weyl geometry is a structure reduction of conformal geometry, and the Einstein-Weyl condition is expressed 
as vanishing of a part of the curvature, namely the trace-free part of the symmetrized Ricci tensor of the Weyl connection.

It turns out that the situation in 5D is similar, and dispersionless integrability can be written as
a zero-curvature condition for a contact parabolic geometry.

\subsection{Subconformal structures on solutions}\label{Sec:1.1}

Let us explain the main ingredients of this article in the case of a scalar  second order PDE, 
which will be fundamental for the  more general case of PDE systems in Section \ref{Sec:3.1}. 
A scalar second order PDE can be expressed as
 \begin{equation}\label{F}
\E\;:\; F(x,u,\p u, \p^2 u)=0
 \end{equation}
for a scalar function $u$ of an independent variable $x$ on a connected
manifold $M$ with $\dim M=d$, where $\p u=(u_i)$ and $\p^2u=(u_{ij})$ denote
partial derivatives of $u$ in local coordinates $x=(x^i)$.  

Let $M_u$ denote the manifold $M$ equipped with a scalar function $u$; 
we may view $M_u$ as the graph of $u$ in $M\times\RR$ or, as we will see later, in the appropriate jet-space.
A tensor on $M_u$ is, by definition, a tensor on $M$, which may also depend, at each $x\in M$, 
on finitely many derivatives of $u$ at $x$.

Let $\ell_F(u)$ be the \emph{linearization} on a solution $u$ of $\E$
 $$
\ell_F(u)v=\left.\frac{d}{d\epsilon}\right|_{\epsilon=0}F(u+\epsilon v),
 $$
which clearly depends on the background solution $u$ if $F$ is nonlinear, and
let $\z_F$ be the \emph{symbol}, i.e.\ the top degree of $\ell_F$ involving only second derivatives
 \begin{equation*}
\z_F = \sum_{i\leq j}\frac{\p F}{\p u_{ij}}\,\p_i\p_j=
\sum_{i,j} \z_{ij}(u)\, \p_i\otimes \p_j.
 \end{equation*}
Invariantly, $\z_F$ defines a section of $S^2T M_u$, hence a quadratic form on
$T^*_x M_u$ for each $x\in M_u$.
Changing the defining function $F$ of $\E$ changes $\z_F$ by a conformal rescaling on $\E$.
Hence, the conformal class of $\z_F$ along $F=0$ is an invariant of $\E$, and so is
the \emph{characteristic variety} $\op{Char}(\E,u)\to M_u$, which is defined as a bundle whose fiber at $x\in M_u$ is
the projective variety
 $$
\op{Char}(\E,u)_x=\{[\theta]\in \PP(T^*_x M_u)\,|\, \z_F(\theta)=0 \}.
 $$
In what follows it is important  that this variety is a quadric, however it is degenerate.
Actually, the following is a direct corollary of the characteristic property of \cite[Theorem 1]{CK},
hidden in the discussion of the coisortopic propery of dLp in loc.cit.

 \begin{theorem}[\cite{CK}]
For integrable determined PDE systems in dimension $d>4$ with quadratic characteristic variety,
the bilinear form $\z_F$ and hence the characteristic variety are degenerate.
In fact, the rank of the symmetric matrix $[\z_{ij}]$ does not exceed 4.
 \end{theorem}

In this paper we assume that $\op{rank}[\z_{ij}]$ is maximal, i.e.\ equal to 4 on generic solutions $u$ of $\E$.
This means that the vector distribution (obtained by contraction)
 $$
\Delta_\E=\langle\sigma_F,T^*M_u\rangle\subset TM_u
 $$
has rank 4. In general, $\Delta_\E$ is a non-integrable and even nonholonomic distribution of rank 4 on $M_u$
for generic solution $u$ of $\E$. Moreover, in what follows we assume that $d=\dim M_u=5$ and that $\Delta_\E$
is a \emph{contact distribution} on $M_u$.

The symmetric bivector $\sigma_F$ is nondegenerate on $\Delta_\E$
and can be inverted to give sub-pseudo-Riemannian structure defined up to scale, i.e.\
subconformal structure $c_\E=[g_F]$ on any solution:
 $$
g_F=(\z_F|_{\Delta_\E})^{-1}\in S^2\Delta_\E^*,\quad\text{where}\quad [g_{ij}] =[\z_{ij}]^{-1}.
 $$
In what follows, we assume that the conformal structure $c_\E$ has neutral signature $(2,2)$.
In other signatures one still can use the same approach via complexification, yet
the real integrability conditions can be more restrictive. For instance, in 4-dimensional Lorentzian conformal structures 
the integrability implies that $F$ is the standard wave operator plus lower order terms.

For any point $x\in M_u$, recall that the neutral signature conformal structure $c_\E|_x$
on a 4-dimensional space $\Delta_\E(x)$  has 2 one-parameter families of isotropic planes 
 that is $\CC\PP^1$ in the complex case and $\RR\PP^1\simeq {S}^1$ in the real case. These are traditionally called $\a$- and $\b$-planes, 
or self-dual or anti-self-dual null planes, which are swapped upon the change of orientation.
The structure $(\Delta_\E,c_\E)$ on $M_u$ is called {\em compatible} if the $\a$-family can be chosen
to consist of Lagrangian planes with respect to the conformally symplectic structure on the contact distribution.
 In this case a generic $\b$-plane would not be Lagrangian. Note that the existence of dLp implies the existence of a 1-parameter family of surfaces whose tangent plane at each point lie in the contact distribution and are null with respect to the subconformal structure.  Since an integral surface of any field of contact planes is necessarily Legendrian, then its tangent planes are Lagrangian with respect to the induced conformal symplectic structure. In other words, the existence of dLp requires  the compatibility condition which will be assumed in what follows.

\subsection{An example of integrable PDE}\label{Sec:1.2}

Consider the following example of 5D version of the heavenly equation wherein $c$ is a constant:
 \begin{equation}\label{5ieq}
F=u_{15}+c\,u_{25}+u_{13}u_{24}-u_{14}u_{23}=0.
 \end{equation}
This equation is a travelling wave reduction of the well-known Takasaki-Plebansky-Przanowski equation in 6D 
\cite{Ta,PP}, and it in turn reduces to the first Plebansky (also known as heavenly), modified heavenly and 
Hussain's equations \cite{Pb,DF}.

The symbol of $F$ is $\z_F=v_1\cdot v_3+v_2\cdot v_4$ where 
$v_1=\p_1$, $v_2=\p_2$,
$v_3=\p_5-u_{23}\p_4+u_{24}\p_3$,
$v_4=c\,\p_5+u_{13}\p_4-u_{14}\p_3$;
 $$
\Delta=\langle v_1,v_2,v_3,v_4\rangle=\op{Ker}(\omega^0),
 $$
 \begin{multline*}
\omega^0=(u_{13}+c\,u_{23})\,dx^3+(u_{14}+c\,u_{24})\,dx^4+(u_{15}+c\,u_{25})\,dx^5\\
=d(u_1+c\,u_3)-(u_{11}+c\,u_{13})\,dx_1-(u_{12}+c\,u_{23})\,dx_2.
 \end{multline*}
Additionally,  on the contact plane $\Delta$ we have
 $$
g=\z_F^{-1}=\omega^1\omega^3+\omega^2\omega^4
 $$
where
 $$
\omega^1=dx^1,\ \omega^2=dx^2,\
\omega^3=\frac{c\,dx^3+u_{14}dx^5} {u_{14}+c\,u_{24}},\
\omega^4=\frac{dx^4+u_{23}dx^5} {u_{13}+c\,u_{23}}.
 $$
The $\beta$-planes $\op{Ker}\{\omega^1-\lambda\omega^2,\omega^4+\lambda\omega^3\}$
are not Lagrangian for $\Omega=d\omega^0$ except for 2 particular values of $\lambda$,
which can be projectively reparametrized to $0$ and $\infty$.

The $\alpha$-planes $\op{Ker}\{\omega^1+\lambda\omega^4,\omega^2-\lambda\omega^3\}
=\langle v_4-\lambda v_1,v_3+\lambda v_2\rangle$
is a family of Lagrangian  planes, and they give rise to the following dispersionless Lax pair for \eqref{5ieq}:
 $$
\hat\Pi=\langle \p_5-u_{23}\p_4+u_{24}\p_3+\lambda \p_2,
c\,\p_5+u_{13}\p_4-u_{14}\p_3-\lambda \p_1\rangle.
 $$

\subsection{Main results}\label{Sec:1.3}

In general, a \emph{dispersionless Lax pair}, referred to as \emph{dLp}, for~\eqref{F} is
a rank one covering system $\hat{\E}$ of $\E$ of a special kind \cite{Za,KV}.
Namely, there is a fiber bundle $\pi\colon \hat M_u\to M_u$ with
connected rank one fibers, and a rank 2 distribution $\hat\Pi\subset T\hat M_u$
such that its Frobenius integrability condition $[\hat\Pi,\hat\Pi]=\hat\Pi$ is a quasilinear overdetermined
system $\hat\E$ whose compatibility is $\E$. Often this latter condition is relaxed to
the claim that $\E$ is a differential corollary of $\hat M_u$, but it is assumed that $\E$ is restored up to
a simple integration/non-localities of potentiation kind; see a discussion in \cite{CK}.
In practice, one can choose generators of $\hat\Pi$ to be linearly independent vector
fields $\hat X$ and $\hat Y$ on $\hat M_u$, whose coefficients depend on finitely many
derivatives of $u$.

A fiber coordinate $\l\colon\hat M_u\to \RR$ is called a \emph{spectral parameter} and it locally
identifies $\hat M_u$ with $M_u\times\RR$. 
We may then write $\hat X=X+m\,\p_\l$, $\hat Y=Y+n\,\p_\l$ where $X,Y$ are $\l$-parametric
vector fields on $M_u$, and a section of $\pi\colon\hat M_u\to M_u$ may be
written $\l=q(x)$ for a function $q\colon M_u\to\RR$. The dLp $\hat\Pi$ then
has the geometric interpretation that $\E$ is the integrability condition for
the existence of many one-parameter families  of foliations  of $M_u$ by surfaces which
at any $x\in M_u$ are tangent to $\Pi=\pi_*(\hat\Pi)$ at $x$, with $\l=q(x)$.


By the mentioned (co)isotropic propery of the dLp,  we may thus identify $\hat M_u$ locally with the $\PP^1$-bundle 
whose fiber over $x\in M_u$ consists of all $\a$-planes. 
In this case, at each $x\in M$, $\Pi$ gives an immersion  $\PP^1\to\mathrm{Gr}(2,\Delta)$. 
Under this identification, with $\Pi$ considered as a $\PP^1$-bundle over $M_u$, 
the pullback $\pi^*\Pi\to\hat M_u$ is the tautological bundle.

Any \emph{Weyl connection} $\nabla$ on $M_u,$ i.e.\ a torsion free connection that preserves the conformal class $c_\E$ 
and depends on finitely many derivatives of $u$, induces a connection on the $\PP^1$-bundle $\hat M_u\to M_u$. 
Hence, it defines a horizontal lift of $\Pi$ to a rank 2 distribution $\hat\Pi_\nabla\subset T\hat M_u$.
It is well-known \cite{Pe,CK} that in 4D  the lift $\hat\Pi_\nabla$ is independent of $\nabla,$ i.e.\
\emph{conformally invariant}; the same applies in higher dimensions provided that $\op{rank}[\sigma_{ij}]=4$.


Now let $(\Delta_\E,c_\E)$ be a compatible subconformal structure on a contact distribution in 5D.
In Section \ref{Sec:2} we will associate to such structure a coframe $\omega^i$
on $M_u$ which is determined up to a projective reparametrization and rescalings, whose group action
depends on a discrete invariant $\op{sgn}\delta=\pm1$ that distinguishes between the CR and para-CR types.
The Cartan-Tanaka prolongation of this structure is of finite type and this implies
the fundamental curvature components. Important for us is a tensor $W=(W_i)_{0}^4$ generalizing the
self-dual part of the Weyl tensor in 4D conformal geometry, which we will describe in detail in Sections
\ref{Sec:2.2}-\ref{Sec:2.4}.

Similar to the situation in 4D, there exists a canonical lift of the $\PP^1$-bundle of $\alpha$-planes, i.e.\ 
the \emph{$\a$-congruence} on $M$, to a rank 2 distribution in 5D. This will be called a \emph{standard dLp}. 
Two Lax pairs are $\E$-equivalent if their restriction to the (infinitely prolonged) equation $\E$ coincide. A dLp is called \textit{nondegenerate} if it is not $\E$-equivalent to a dispersionless pair, for which the Frobenius integrability condition holds identically, also referred to \textit{off shell,} i.e. not as a corollary of $\E$, c.f. \cite[Definition 5]{CK}.

 \begin{theorem}\label{Th1}
Let $\E: F=0$ be a determined PDE system in 5D, whose characteristic variety $\op{Char}(\E)$ is a bundle of
quadric hypersurfaces of maximal rank 4 in $\PP\,T^*M_u$, and such that the subconformal structure
on contact distribution $(\Delta_\E,c_\E)$ on $M_u$ is compatible for (almost) every solution $u$.
Then any nondegenerate dLp $\hat\Pi$ is $\E$-equivalent to a standard dLp $\hat\Pi_\nabla$.
 \end{theorem}
 
Our  main result establishes an equivalence between the
dispersionless integrability of $\E$ and the zero-curvature property of $c_\E$.

 \begin{theorem}\label{Th2}
Under the same assumptions as in Theorem \ref{Th1}, $\E$ is integrable by a nondegenerate dLp
if and only if the zero-curvature condition holds nontrivially on solutions of $\E$.
 \end{theorem}

Finally let us give a presentation of master-equation for dispersionless integrable equations in 5D
with contact characteristic distribution. Here by ``master'' we mean two following properties:
 \begin{itemize}
\item Any integrable background subconformal geometry in 5D is locally described by this equation, 
as a particular solution of the PDE on the functional parameters;
\item Any integrable dispersionless equation with the specified conditions on the characteristic variety
is a reduction of this equation, meaning that all solutions of the given PDE are locally reparametrizations of
the solutions of the master equation.
 \end{itemize}
Such master equation in 3D was identified in \cite{DFK} with the Manakov-Santini equation,
governing Einstein-Weyl structures, and in 4D a determined PDE was derived in \cite{DFK},
governing the self-dual conformal structures.

This can be considered as the quotient problem of the space of all background integrable geometries in 5D
by the pseudogroup of local diffeomorphisms, where the residual gauge transformations 
can be eliminated via the differential invariants approach, as was done in \cite{KS1,KS2} for the EW and SD systems.

Leaving the precise form of the master-equation to Section \ref{Sec:3} let us preview the result.

 \begin{theorem}\label{Th3}
General zero-curvature integrable subconformal structures on contact distributions in 5D are
given locally by the contact form $dr-p\,dx-q\,dy$ and the conformal metric
 \begin{align*}
g=&\, (dx + (u+v)\,dy)\cdot(dq - (u-v)\,dp + w\,dx + (z-w(u-v))\,dy)\\
 +&\, (dx + (u-v)\,dy)\cdot(dq - (u+v)\,dp + w\,dx + (z-w(u+v))\,dy),
 \end{align*}
where the functions $u,v,w,z$ of variables $x,y,p,q,r$ satisfy PDEs $\{W_i=0\}_{i=0}^3$ given by \eqref{W0123}.
This system is itself integrable via a dispersionless Lax pair (\ref{MEdLp1}-\ref{MEdLp3}). 
The local moduli space of zero-curvature integrable subconformal structures is parametrized by 8 functions of 4 variables.
 \end{theorem}

\smallskip

{\bf Structure of the paper.} 
In Section \ref{Sec:2} we  investigate the background geometry in dimension five,
i.e.\ the subconformal contact structure on $M_u$ with appropriate compatibility conditions, 
express the structure equations and  fundamental  invariants and link the integrability via dLp to the zero-curvature condition.  
In Section \ref{Sec:3} we describe the  systems of PDEs to which our results extend,
introduce the necessary tools from jet-machinery, and prove the main results.
In Section \ref{Sec:4} we  relate our results to other geometries via twistorial techniques and symmetry reductions. 
We conclude in Section  \ref{Sec:5} with an overview of the main results and discuss possible generalizations. In the Appendix we classify integrable parabolic background geometries, in addition to the geometry discussed in this paper and the previously known parabolic geometries in 4D and 3D,
namely Cartan geometries of type  $(SO(3,3),P_1)$ and $(SO(2,3),P_1)$, the latter of which is equipped with a choice of Weyl structure.

\medskip

{\bf Conventions.} 
In this article all manifolds are real and smooth. Given a (system of) differential equation we work locally away from singularities.   
For a foliation of a manifold, we restrict to open sets in which the leaf space of the foliation is a smooth manifold. 
Given a set of 1-forms $\{\a^1,\dots,\a^n\}$ on a manifold $M$, their span is denoted by $I=\langle\a^i\rangle_{i=1}^n$ 
and their kernel (annihilator) is denoted by $\op{Ker} I$ or $I^\perp$ interchangeably. 
For a distribution $\Delta\subset TM$ its annihilator will be denoted $I=\op{Ann}(\Delta)$.
For a pseudo-Riemannian metric $g\in S^2 T^* M$ the corresponding conformal structure is denoted by either $[g]$ or $c_g$. 
Given a bundle $\cG\to M$ with a coframe  $\langle\a^1,\dots,\a^n,\b^1,\dots,\b^k\rangle$ on $\cG$ 
the dual frame is denoted by $\langle\p_{\a^1},\dots,\p_{\a^n},\p_{\b^1},\dots,\p_{\b^k}\rangle$. Furthermore, if   
$\langle\a^i\rangle_{i=1}^n$ are  semi-basic with respect to a fibration $\cG\to M$,  then  given a function $f\colon\cG\to \RR,$ the iterative coframe derivatives of $f$ are defined as
 \[
f_{;i}=\partial_{\alpha^i}\lrcorner\, df,\quad f_{;ij}=\partial_{\alpha^j}\lrcorner\, df_{;i},\quad \text{etc.}
 \]

\bigskip

\textsc{Acknowledgment.}
We are grateful to E.\,Ferapontov and D.\,Sykes for helpful discussions and comments. BK acknowledges hospitality
of IMPAN Warsaw, OM acknowledges hospitality of UiT Troms\o, where parts of this work have been performed.

\section{Geometry of subconformal structures in 5D}\label{Sec:2}

Here we discuss compatible subconformal structures on contact 5D manifolds $(M,\Delta)$.
Let $\omega^0\in\op{Ann}(\Delta)\setminus0$ be a contact form. Then $\Omega=d\omega^0|_\Delta$
defines a conformal symplectic structure. We will operate with structures on $\Delta$ via a coframe $\omega^i$,
$1\leq i\leq 4$, modulo $\omega^0$. The dual frame on $\Delta$ will be denoted by
$\p_{\omega^1},\p_{\omega^2},\p_{\omega^3},\p_{\omega^4}$, and it is complemented by the transversal
vector $\p_{\omega^0}$ (which may be chosen to be the Reeb vector field but we do not require it).

\subsection{Compatibility}\label{Sec:2.1}

A subconformal structure on $(M,\Delta)$ is the conformal class of a nondegenerate bilinear form $g\in\Gamma(S^2\Delta^*)$, which will be assumed of neutral signature $(2,2)$. In null-diagonal coframe
we get:
 \begin{equation}\label{gg}
g=\omega^1\omega^3+\omega^2\omega^4.
 \end{equation}
There are two  $\PP^1$-bundles of contact planes, denoted as $\mathcal{A},\mathcal{B}\subset\op{Gr}_2(\Delta)$
and referred to as $\a$- and $\b$-planes respectively, which are totally null with respect to $[g]$, i.e.\ at
every point $x\in M$ 
\begin{equation}\label{eq:alpha-beta-planes}
\begin{aligned}
\mathcal{A}_x:=& \left\{\Pi=\op{Ker}\{\a_0\omega^1-\a_1\omega^2,\a_0\omega^4+\a_1\omega^3,\omega^0\}
\,|\, [\alpha_0:\alpha_1]\in\PP^1 \right\},\\
\mathcal{B}_x:=& \left\{\Pi=\op{Ker}\{\b_0\omega^1-\b_1\omega^4,\b_0\omega^2+\b_1\omega^3,\omega^0\}
\ \,|\  [\beta_0:\beta_1]\in\PP^1 \right\}.
\end{aligned}
 \end{equation}

With respect to the Hodge operator $*$ on $(\Delta,g)$ these are self-dual and anti-self-dual planes, 
respectively, and they are swapped upon the change of orientation on $\Delta$.
In what follows, we choose to focus on the family of $\a$-planes.

  \begin{definition}\label{def:subconf-contact}
A subconformal contact structure $(M,\Delta,[g])$ is {\em compatible} if all $\alpha$-planes of $[g]$
 \begin{equation}\label{gPi}
\Pi=\langle\p_{\omega^1}+\lambda\p_{\omega^2},\lambda\p_{\omega^3}-\p_{\omega^4}\rangle,
\quad \lambda=\tfrac{\a_0}{\a_1}\in\RR\cup\infty,
 \end{equation}
are Lagrangian with respect to the induced conformal symplectic structure $[\Omega]$ on $\Delta$.
  \end{definition}

We also call $\Omega$ compatible under the same conditions with respect to $g$ on $\Delta$.

 \begin{lemma}
A compatible symplectic structure has the form
 $$
\Omega=p\,\omega^1\wedge\omega^2+q\,(\omega^1\wedge\omega^3+\omega^2\wedge\omega^4)
+r\,\omega^3\wedge\omega^4
 $$
with a relative invariant $\delta=q^2-pr\neq0$.
 \end{lemma}

 \begin{proof}
Indeed, writing $\Omega=\sum_{i<j}c_{ij}\omega^i\wedge\omega^j$ modulo $\omega^0$,
evaluating this on $\Pi$ and expanding by $\l$ gives $c_{23}=0$, $c_{13}=c_{24}$, $c_{14}=0$.
Then $\frac12\Omega^2=-\delta\op{vol}_g$, where
$\op{vol}_g=\omega^1\wedge\omega^2\wedge\omega^3\wedge\omega^4$.
 \end{proof}

 \begin{proposition}
If $\Omega$ is compatible with \eqref{gPi} then it is conformally equivalent to one of the forms
 \begin{gather}
\delta>0:\ \Omega=\omega^1\wedge\omega^3+\omega^2\wedge\omega^4,\label{Om-}\\
\delta<0:\ \Omega=\omega^1\wedge\omega^2+\omega^3\wedge\omega^4.\label{Om+}
 \end{gather}
 \end{proposition}

 \begin{proof}
Let us introduce an operator $J=g^{-1}\Omega$ on $\Delta$.
Normalizing $|\delta|=1$ by rescaling $\Omega$ is equivalent to the spectrum of $J$ belonging to the unit circle $S^1\subset\CC$.
More explicitly, if $\delta=+1$ then eigenvectors of $J$ are $\pm1$ and
if $\delta=-1$ then eigenvectors of $J$ are $\pm i$, in both cases both eigenvalues have multiplicity 2
and the operator $J$ is semi-simple and is related to $g$ by
 $$
g(Jv,w)+g(v,Jw)=0\quad\forall v,w\in\Delta.
 $$

In the case $\delta=+1$, $J^2=\1$ and the $g$-null $\Omega$-Lagrangian planes are generated by
$J$-eigenvectors, and so are either eigenspaces $L_\pm=E_J(\pm1)$ or
belong to the family $\Pi$ generated by an orthogonal pair of vectors from $L_-$ and $L_+$.
Choosing null-orthogonal basis of eigenvectors of $J$ as in \eqref{gg} we get the required formula \eqref{Om-}.

In the case $\delta=-1$, $J^2=-\1$ and the $g$-null $\Omega$-Lagrangian planes form a 1-parameter
family $\Pi$, but no such singular planes as in the case $\delta=+1$ exist.
Then a pair of $J$-invariant null planes yields null-orthogonal basis \eqref{gg} and the required formula
\eqref{Om+} follows.
 \end{proof}

Note that in the family of $\b$-planes given by
 \begin{equation*}
\Pi'= \langle\p_{\omega^1}+\lambda\p_{\omega^4},\p_{\omega^2}-\lambda\p_{\omega^3}\rangle,
\quad \lambda=\tfrac{\b_0}{\b_1}\in\RR\cup\infty,
 \end{equation*}
$\Omega$-Lagrangian planes for $\delta>0$ correspond to $\lambda=0$ and $\infty,$ 
i.e.\ $L_-=\langle\p_{\omega^1},\p_{\omega^2}\rangle$ and $L_+=\langle\p_{\omega^3},\p_{\omega^4}\rangle$,
while for $\delta<0$ the equation for such planes is $\lambda^2+1=0$ and it has no real solutions.

 \begin{corollary}\label{crGL}
A compatible subconformal structure $(M,\Delta,[g])$ possesses an adapted coframe $\{\omega^i\}_{i=0}^4$,
in which $\Delta=\{\omega^0=0\}$ and formulae \eqref{gg} together with either  \eqref{Om-} or \eqref{Om+} hold.
The part $\{\omega^i\}_{i=1}^4$ modulo $\omega^0$ is defined up to an action of 
$\RR_\times GL(2,\RR)\subset\op{End}(\Delta)$.
 \end{corollary}

 \begin{proof}
The first statement is the direct corollary of the proposition. The second follows from the formulae
\eqref{gg}-\eqref{Om-}-\eqref{Om+}. In matrix terms the action is given as follows:
 \begin{equation}\label{AA}
A_{+}=\begin{pmatrix}
a_{11} & a_{12} & 0 & 0 \\ a_{21} & a_{22} & 0 & 0 \\
0 & 0 & b\,a_{22} & -b\,a_{21} \\ 0 & 0 & -b\,a_{12} & b\,a_{11}
\end{pmatrix},\qquad
A_{-}=\begin{pmatrix}
a_{11} & a_{12} & -b\,a_{12} & b\,a_{11} \\ a_{21} & a_{22} & -b\,a_{22} & b\,a_{21} \\
b\,a_{21} & b\,a_{22} & a_{22} & -a_{21} \\ -b\,a_{11} & -b\,a_{12} & -a_{12} & a_{11}
\end{pmatrix}
 \end{equation}
where the first matrix corresponds to $\delta>0$ and the second to $\delta<0$.

Note that those algebras preserve 
$J_+=-\p_{\omega^1}\otimes\omega^1-\p_{\omega^2}\otimes\omega^2
+\p_{\omega^3}\otimes\omega^3+\p_{\omega^4}\otimes\omega^4$ and 
$J_-=\p_{\omega^1}\otimes\omega^4-\p_{\omega^4}\otimes\omega^1
-\p_{\omega^2}\otimes\omega^3+\p_{\omega^3}\otimes\omega^2$ respectively.
 \end{proof}

Formula \eqref{AA} shows the action of $GL(2,\RR)$ on $\Delta\simeq\RR^4$, 
which is reducible for $\delta>0$ and irreducible for $\delta<0$.
It induces the action of $PSL(2,\RR)$ on the projective line $\PP^1$ of $\a$-planes.

\subsection{Curvature in subconformal geometry: naïve approach}\label{Sec:2.2} 

We consider the case $\delta>0$ given by \eqref{gg} and \eqref{Om-} and the case
$\delta<0$ given by \eqref{gg} and \eqref{Om+} simultaneously.

The pre-image of  $\alpha$-planes  to  $\hat{M}$,  which
is the total space of the $\PP^1$-bundle with fibers  $\mathcal{A}_x$ in \eqref{eq:alpha-beta-planes}, is given by
 $$
\pi^{-1}_*\Pi=\langle\p_{\omega^1}+\lambda\p_{\omega^2},\lambda\p_{\omega^3}-\p_{\omega^4},\p_\lambda\rangle.
 $$
where $\lambda$ is a local coordinate along the fibers. Given the compatibility condition, this distribution has growth vector $(3,5,6)$ and hence it possesses
the radical    (dLp) uniquely given by the condition
$[\hat{\Pi},\hat{\Pi}]\subset\pi^{-1}_*\Pi$, and so we get
 $$
\hat{\Pi}=\sqrt{\pi_*^{-1}\Pi}=
\langle\p_{\omega^1}+\lambda\p_{\omega^2}+m\p_\lambda,
\lambda\p_{\omega^3}-\p_{\omega^4}+n\p_\lambda\rangle.
 $$
This can be considered as a lift of the configuration $\Pi$ on $M$, and the coefficients $m,n$ are uniquely
determined as follows. Let the structure equation of the frame be
 $$
d\omega^k=-\frac12c_{ij}^k\omega^i\wedge\omega^j\ \Leftrightarrow\
[\p_{\omega^i},\p_{\omega^j}]=c_{ij}^k\p_{\omega^k}.
 $$
We compute
 \begin{multline*}
[\p_{\omega^1}+\lambda\p_{\omega^2}+m\p_\lambda,
\lambda\p_{\omega^3}-\p_{\omega^4}+n\p_\lambda]\\
=\lambda c_{13}^k\p_{\omega^k}-c_{14}^k\p_{\omega^k}+\lambda^2 c_{23}^k\p_{\omega^k}
-\lambda c_{24}^k\p_{\omega^k}+m\p_{\omega^3}-n\p_{\omega^2}\,\op{mod}\langle\p_\lambda\rangle\\
\equiv (-\lambda^2 c_{13}^1+\lambda c_{13}^2+\lambda c_{14}^1-c_{14}^2
-\lambda^3 c_{23}^1+\lambda^2 c_{23}^2+\lambda^2 c_{24}^1-\lambda c_{24}^2-n)\p_{\omega^2}\\
+(\lambda c_{13}^3+\lambda^2 c_{13}^4-c_{14}^3-\lambda c_{14}^4
+\lambda^2 c_{23}^3+\lambda^3 c_{23}^4-\lambda c_{24}^3-\lambda^2 c_{24}^4+m)\p_{\omega^3}
  \,\op{mod}\langle\p_\lambda\rangle
 \end{multline*}
whence
 \begin{gather*}
m=-\lambda^3 c_{23}^4+\lambda^2(c_{24}^4-c_{23}^3-c_{13}^4)
+\lambda(c_{24}^3+c_{14}^4-c_{13}^3)+c_{14}^3,\\
n=-\lambda^3 c_{23}^1+\lambda^2(c_{24}^1+c_{23}^2-c_{13}^1)
-\lambda(c_{24}^2-c_{14}^1-c_{13}^2)-c_{14}^2.
 \end{gather*}
Now the curvature of the subconformal structure, considered as the obstruction to integrability of dLp,
is $d\lambda([\p_{\omega^1}+\lambda\p_{\omega^2}+m\p_\lambda,
\lambda\p_{\omega^3}-\p_{\omega^4}+n\p_\lambda])$, which equals
 \begin{equation}\label{WW-}
W= \p_{\omega^1}(n)+\lambda\p_{\omega^2}(n)-\lambda\p_{\omega^3}(m)+\p_{\omega^4}(m)
+m\,n_\lambda-n\,m_\lambda.
 \end{equation}
where $m_\lambda=\partial_\lambda m.$ This is clearly a quartic in $\lambda$ responsible for Frobenius integrability of dLp.
In fact, with the notations
 \begin{align*}
& m_0=c_{14}^3, &
& m_1=c_{24}^3+c_{14}^4-c_{13}^3, &
& m_2=c_{24}^4-c_{23}^3-c_{13}^4, &
& m_3=-c_{23}^4, \\
& n_0=-c_{14}^2, &
& n_1=c_{13}^2+c_{14}^1-c_{24}^2, &
& n_2=c_{24}^1+c_{23}^2-c_{13}^1, &
& n_3=-c_{23}^1,
 \end{align*}
we have $W=W_0+W_1\lambda+W_2\lambda^2+W_3\lambda^3+W_4\lambda^4$, where
 \begin{gather*}
W_0= \p_{\omega^1}(n_0)+\p_{\omega^4}(m_0)+m_0\,n_1-m_1\,n_0,\\
W_1= \p_{\omega^1}(n_1)+\p_{\omega^2}(n_0)-\p_{\omega^3}(m_0)+\p_{\omega^4}(m_1)
+2(m_0n_2-m_2n_0),\\
W_2= \p_{\omega^1}(n_2)+\p_{\omega^2}(n_1)-\p_{\omega^3}(m_1)+\p_{\omega^4}(m_2)
+3(m_0n_3-m_3n_0) + m_1n_2-m_2n_1,\\
W_3= \p_{\omega^1}(n_3)+\p_{\omega^2}(n_2)-\p_{\omega^3}(m_2)+\p_{\omega^4}(m_3)
+2(m_1n_3-m_3n_1),\\
W_4= \p_{\omega^2}(n_3)-\p_{\omega^3}(m_3)
+m_2n_3-m_3n_2.
 \end{gather*}

\subsection{Associated parabolic geometry}\label{Sec:2.3}

A compatible subconformal structure can be equivalently described as a parabolic geometry of
type $(A_3,P_{13})$. We refer to \cite{CS} for the basics of parabolic geometries and to many
examples, including those related to ours.

 \begin{remark}
Among $A_3$ type parabolic geometries the following are well-known and considered in \cite{CS}:
3D projective geometry has type $(A_3,P_1)$,
4D conformal geometry has type $(A_3,P_2)$,
geometry of systems of 2nd order ODEs with 2 dependent variables has type $(A_3,P_{12})$ and vanishing
torsion of the lowest weight,
CR structures in 5D or its para-version integrable Legendrian structures have type $(A_3,P_{13})$
and vanishing torsion. This latter case is different from ours, which allows torsion but requires vanishing curvature.
 \end{remark}

Here we only note that the underlying geometric structure for $(A_3,P_{13})$, in the complex case,
is a contact distribution $\Delta$ on 5-manifold $M$, and this distribution is split $\Delta=L_-\oplus L_+$ into
the sum of two Lagrangian subbundles, which are thus conformally dual to each other.
The matrix form of the parabolic subalgebras $\fp=\op{Lie}(P)$ corresponds to the part of
non-negative grading, as follows (we also show other parabolics that will be relevant later):
 $$
\underbrace{\begin{bmatrix} \gc{0} & \gc{+1} & \gc{+1} & \gc{+1} \\ \gc{-1} & \gc{0} & \gc{0} & \gc{0} \\
\gc{-1} & \gc{0} & \gc{0} & \gc{0} \\ \gc{-1} & \gc{0} & \gc{0} & \gc{0} \end{bmatrix}}
_{\fp_1=\g_0\oplus\g_1}\quad
\underbrace{\begin{bmatrix} \gc{0} & \gc{0} & \gc{+1} & \gc{+1} \\ \gc{0} & \gc{0} & \gc{+1} & \gc{+1} \\
\gc{-1} & \gc{-1} & \gc{0} & \gc{0} \\ \gc{-1} & \gc{-1} & \gc{0} & \gc{0} \end{bmatrix}}
_{\fp_2=\g_0\oplus\g_1}\quad
\underbrace{\begin{bmatrix} \gc{0} & \gc{+1} & \gc{+2} & \gc{+2} \\ \gc{-1} & \gc{0} & \gc{+1} & \gc{+1} \\
\gc{-2} & \gc{-1} & \gc{0} & \gc{0} \\ \gc{-2} & \gc{-1} & \gc{0} & \gc{0} \end{bmatrix}}
_{\fp_{12}=\g_0\oplus\g_1\oplus\g_2}\quad
\underbrace{\begin{bmatrix} \gc{0} & \gc{+1} & \gc{+1} & \gc{+2} \\ \gc{-1} & \gc{0} & \gc{0} & \gc{+1} \\
\gc{-1} & \gc{0} & \gc{0} & \gc{+1} \\ \gc{-2} & \gc{-1} & \gc{-1} & \gc{0} \end{bmatrix}}
_{\fp_{13}=\g_0\oplus\g_1\oplus\g_2}
 $$
There are three different real versions, corresponding to the following real groups $G$ of type $A_3$:
$SL(4,\RR)$, $SU(1,3)$ and $SU(2,2)$, each of which has a parabolic subgroup of type $P_{1,3}$.

However real compatible subconformal structures correspond only to the first and the last ones.
The middle real parabolic geometry has induced conformal structure on $\Delta$ of definite signature,
hence does not possess null 2-planes, which can be taken as dLp candidates. The other two cases have the following
notations as crossed Dynkin diagrams:
 $$
\mathfrak{p}_{1,3}\subset\mathfrak{sl}(4,\RR):
{
 \begin{tiny}
 \begin{tikzpicture}[scale=0.8,baseline=-3pt]
\bond{0,0}; \bond{1,0}; \DDnode{x}{0,0}{}; \DDnode{w}{1,0}{}; \DDnode{x}{2,0}{};
 \useasboundingbox (-.9,-.4) rectangle (2.4,0.4); 
 \end{tikzpicture}
 \end{tiny}
 }\qquad\qquad
\mathfrak{p}_{1,3}\subset\mathfrak{su}(2,2):
{
 \begin{tiny}
 \begin{tikzpicture}[scale=0.8,baseline=-3pt]
\bond{0,0}; \bond{1,0}; \DDnode{x}{0,0}{}; \DDnode{w}{1,0}{}; \DDnode{x}{2,0}{};
\node (A) at (0,0.1) {}; \node (B) at (2,0.1) {}; \path[<->,font=\scriptsize,>=angle 90] (A) edge [bend left] (B);
 \useasboundingbox (-.9,-.4) rectangle (2.4,0.4); 
 \end{tikzpicture}
 \end{tiny}
 }
 $$

In the first case, corresponding to $\delta>0$, the parabolic geometry on $M^5$ results in a splitting
of $\Delta$ and the conformal duality $L_+\simeq L_-^*$. Conversely, given such geometric structure
we define the subconformal structure via pairing null planes $L_-$ and $L_+$. The
configuration of $\alpha$-planes is restored as $\Pi=\ell_\lambda^-\oplus\ell_\lambda^+$, where
$L_-\supset\ell_\lambda^-\perp\ell_\lambda^+\subset L_+$ (so $\ell_\lambda^-$ determines $\ell_\lambda^+$).
The splitting can be encoded via an almost product structure on the contact distribution $\Delta$ given by
$J|_{L_\pm}=\pm\op{Id}_\Delta$. With no requirement of integrability of distributions $L_\pm$,
this geometry is almost para-CR.

Similarly, in the last case, corresponding to $\delta<0$, the parabolic geometry on $M^5$ reads off the splitting
of $\Delta^\CC$ into a pair of complex conjugated 2-distributions 
$L_{10}$ and $L_{01}=\,\overline{\!L_{10}\!\mathstrut}\,$ 
(holomorphic and anti-holomorphic parts of the complexified contact distribution).
This can be encoded via an almost complex structure $J$ on $\Delta$, compatible with the conformally
symplectic structure. However we again do not require integrability of the complex distributions
(or vanihsing of the Nijenhuis tensor), hence this geometry is almost CR of split Levi signature.
Thus, we proved:

 \begin{proposition}\label{prop:paraCR-subconformal-one-to-one}
There is a bijective correspondence in 5D between compatible subconformal contact structures
($\delta>0$, resp, $\delta<0$) and $(A_3,P_{13})$ type parabolic geometries
(Legendrian contact structures, resp, almost CR structures of split Levi signature).
 \end{proposition}

Now we would like to define the notion of a Cartan geometry. Given Lie groups $P\subset G$ let denote their Lie algebras by $\fp$ and $\g,$ respectively.  A \emph{Cartan geometry} $(\mathcal{G}\to M,\psi)$ of type $(G,P)$  is given by a (right) principal $P$-bundle $\tau\colon\mathcal{G}\to M$ together with
a \emph{Cartan connection} $\psi\in\Omega^1(\mathcal{G},\g)$,
i.e a $\g$-valued $1$-form  on $\mathcal{G}$ with the following properties:
 \begin{itemize}
 \item $\psi$ is $P$-equivariant, i.e.  $r_g^*\psi=\mathrm{Ad}_{g^{-1}}\circ\psi$  for   any $g\in P$,
 \item  $\psi$ maps    fundamental   vector  fields to  their generators, i.e. $\psi(\zeta_X)=X$ for any $X\in\p$,
 \item  $\psi$  defines an    isomorphism $\psi\colon T_u\mathcal{G}\to \g$ for any $u\in\mathcal{G}.$
\end{itemize}
The \emph{curvature} of a Cartan connection $\psi$ is  the 2-form $\Psi\in\Omega^2(\mathcal{G},\g)$ defined as
\[\Psi(X,Y)=d\psi(X,Y)+[\psi(X),\psi(Y)]\]
for $X,Y\in\Gamma(T\mathcal{G}).$ 
 
Every parabolic geometry possesses a Cartan connection $\psi\in\Omega^1(\mathcal{G},\g)$
on the principal bundle $\mathcal{G}$ over the base $M$, where $\g=\op{Lie}(G)$ is the Lie algebra of
the corresponding Lie group discussed above.  The curvature of this connection
 $$
K=d\psi+\tfrac12[\psi,\psi]\in\Omega^2(\mathcal{G},\g)
 $$
can be identified with the curvature function
$\kappa:\mathcal{G}\to\Lambda^2\fp_+\otimes\g$ via the Killing form identification
$\g_-^*=(\g/\fp)^*=\fp_+$, where $\g=\g_-\oplus\g_0\oplus\g_+$
is the grading corresponding to the choice of parabolic $\fp=\g_0\oplus\g_+$ and $\fp_+=\g_+$.
The normality condition $\p^*\kappa=0$, where
$\p^*:\Lambda^2\fp_+\otimes\g\to\fp_+\otimes\g$ is the Kostant codifferential,
uniquely determines the Cartan connection \cite{CS}.

The harmonic curvature $\kappa_H$ is the quotient part of $\kappa$ taking values in the $\g_0$ submodule
$\op{Ker}(\Box)=\dfrac{\op{Ker}(\p^*)}{\op{Im}(\p^*)}$ of $\Lambda^2\g_-^*\otimes\g$,
where $\square=\p\p^*+\p^*\p$ is the Kostant Laplacian \cite{Ko}.
This $\kappa_H$ uniquely restores $\kappa$ via invariant differentiations, and is a simpler object,
as it takes values in the Lie algebra cohomology $H^2_+(\g_-,\g)$, where the subscript ``$+$''
indicates positive homogeneity with respect to the grading element $Z\in\g_0$.

Computation of this cohomology is straightforward from the Kostant's version of the Bott-Borel-Weyl theorem \cite{Ko}.
For the complex Lie algebra $A_3$ and its parabolic subalgebra $\fp_{13}$
we have $\g_0=\CC\oplus\mathfrak{sl}(2,\CC)\oplus\CC$, $\g_-=\g_{-2}\oplus\g_{-1}=\mathfrak{heis}(5)$,
and the cohomology $H^2_+$ decomposes into $\g_0$-irreps as follows:
 \begin{equation}\label{H2+}
H^2_+(\g_-,\g)=\mathbb{V}_1^{(2,-1)}\oplus\mathbb{V}_1^{(-1,2)}\oplus\mathbb{V}_4^{(1,1)},
 \end{equation}
where $\mathbb{V}_s$ indicates $\g_0^{ss}=\mathfrak{sl}(2,\CC)$ module and superscripts show the
weight with respect to $\mathfrak{z}(\g_0)=\CC\oplus\CC$.
The first two summands correspond to the torsion $\tau_{\pm}$ and can be identified with
$\Lambda^2L_-^*\otimes L_+$ and $\Lambda^2L_+^*\otimes L_-$.
The last module corresponds to the curvature wherein $W$ takes values.
We thus have:
 \begin{proposition}\label{kH}
The harmonic curvarure of a compatible subconformal structure splits as above
$\kappa_H=\tau_-+\tau_++W$ and the zero-curvature condition is $W=0$.
 \end{proposition}

The correspondence space, as the total space of the  $\mathbb{P}^1$-bundle $\pi:\hat{M}^6\to M^5$
(with $\lambda$ coordinate in fibers), is a parabolic geometry of type $(A_3,P_{123})$ and its
underlying rank 3-distribution $\pi_*^{-1}\Pi$ with growth vector $(3,5,6)$ has the radical $\hat\Pi=\sqrt{\pi_*^{-1}\Pi}$.
The harmonic curvature of the lifted structure on $\hat M^6$ has 5 irreducible components 
 \begin{equation}\label{htkH}
\hat\kappa_H=\varsigma_-+\varsigma_++\tau_-+\tau_++W,
 \end{equation}
all of which are torsion, i.e. the 2-cocycles take values in the corresponding $\g_-$. Because the distribution is a lift of a $(A_3,P_{13})$ structure, one has $\varsigma_\pm=0$.
The zero-curvature condition $W=0$ corresponds to integrability of $\hat\Pi$, so there is
a local quotient, $\hat M^6\to\mathcal{T}^4$, where $\mathcal T^4$ is referred to as the \textit{twistor space}. However  the induced structure on $\mathcal T$ is not a parabolic geometry of type $(A_3,P_2),$ i.e. a conformal structure, unless $\tau_\pm=0$, i.e.\ unless
the subconformal contact structure is flat.

In the real split (para CR) case $\g=\mathfrak{sl}(4,\RR)$ the above expressions hold literally by changing 
$\CC$ to $\RR$ with $\g_0^{ss}=\mathfrak{sl}(2,\RR)$ etc.
In the real CR case $\g=\mathfrak{su}(2,2)$ the first two summands of \eqref{H2+} form an irreducible
module over $\g_0^{ss}=\mathfrak{su}(1,1)\simeq\mathfrak{sl}(2,\RR)$. Thus, $\tau_-+\tau_+$ in Proposition \ref{kH} 
is the indecomposable torsion, and similarly in \eqref{htkH} curvature components
$\varsigma_-+\varsigma_+$ and $\tau_-+\tau_+$ are indecomposable.

\subsection{Structure equations and fundamental invariants}\label{Sec:2.4}

The equivalence method for subconformal compatible contact structures, reformulated as Cartan geometries 
$(\mathcal{G}\to M,\psi)$ of type $(A_3,P_{13})$, yields the full curvature $K$ of the problem. 
For our purposes it suffices to compute only selected entries of the Cartan curvature,
which we do using the structure equations.

In this section, for  brevity, we consider only the almost para-CR case,
corresponding to $\g=\mathfrak{sl}(4,\RR)$. The CR case can be treated similarly.
The Cartan connection and curvature have the following form:
 \begin{equation}\label{eq:A3P123}
\psi=  \begin{pmatrix}
    \phi_2-\rho\! & \xi_2 & \xi_1 & \xi_0\\
    \omega^2 & \!\phi_1-\rho\! & \xi_5 & \xi_4\\
    \omega^1 & \omega^5 & \!\phi_0-\rho\! & \xi_3\\
    \omega^0 & \omega^4 & \omega^3 & \rho
  \end{pmatrix},
	\
\Psi=d\psi+\psi\wedge\psi=  \begin{pmatrix}
    \Phi_2-R\! & \Xi_2 & \Xi_1 & \Xi_0\\
    \Omega^2 & \!\Phi_1-R\! & \Xi_5 & \Xi_4\\
    \Omega^1 & \Omega^5 & \!\Phi_0-R\! & \Xi_3\\
    0 & \Omega^4 & \Omega^3 & R
  \end{pmatrix}
 \end{equation}
where $\rho=\half(\phi_0+\phi_1+\phi_2)$ and  $R=\half(\Phi_0+\Phi_1+\Phi_2)$. 

Imposing the normality conditions and Bianchi identities, modulo $\{\omega^0\}$ one has
\begin{equation}\label{eq:CartanCurvature}
 \begin{aligned}
\Omega^1&\equiv T_-^1\omega^3\w\omega^4,\qquad \Omega^2\equiv T_-^2\omega^3\w\omega^4,\qquad \Omega^3\equiv  T_+^1\omega^1\w\omega^2,\qquad \Omega^4\equiv   T_+^2\omega^1\w\omega^2,\\
    \Omega^5&\equiv -\tfrac 16(T^1_-T^1_++T^2_-T^2_+-W_2)\omega^1\w\omega^4+\tfrac 14(T_+^2T_-^1-W_3)\omega^1\w\omega^3+ \tfrac 14(T_+^2T_-^1+W_3)\omega^2\w\omega^4\\
    &\ -W_4\omega^2\w\omega^3 +x_6\omega^1\w\omega^2+x_4\omega^3\w\omega^4,\\
    \Xi^5&\equiv   -\tfrac 16(T^1_-T^1_++T^2_-T^2_+-W_2)\omega^2\w\omega^3+ \tfrac 14(T_+^1T_-^2-W_1)\omega^2\w\omega^4+ \tfrac 14(T_+^1T_-^2+W_1)\omega^1\w\omega^3\\
    &\ -W_0\omega^1\w\omega^4 +x_2\omega^3\w\omega^4+x_8\omega^1\w\omega^2,\\
    \Phi_0&\equiv  \tfrac{1}{24}(T^1_-T^1_+-5T^2_-T^2_+-4W_2)\omega^1\w\omega^3 - \tfrac{1}{24}(T^2_-T^2_+-5T^1_-T^1_+-4W_2)\omega^2\w\omega^4\\
    &\ +\tfrac 14 W_1\omega^1\w\omega^4-\tfrac 14 W_3\omega^2\w\omega^3 -x_{13}\omega^1\w\omega^2+x_{16}\omega^3\w\omega^4,\\
    \Phi_1&\equiv  -\Phi_0 -(x_{13}+x_{15})\omega^1\w\omega^2+(x_{14}+x_{16})\omega^3\w\omega^4,\\
    \Phi_2&\equiv  \tfrac{1}{12}(T^1_-T^1_+-5T^2_-T^2_+)\omega^1\w\omega^3 + \tfrac{1}{12}(T^2_-T^2_+-5T^1_-T^1_+)\omega^2\w\omega^4+\tfrac 12 T^1_+T^2_-\omega^1\w\omega^4\\
    &\ +\tfrac 12 T^1_-T^2_+\omega^2\w\omega^3 -(x_{13}+x_{15})\omega^1\w\omega^2 -(x_{14}+x_{16})\omega^3\w\omega^4,
  \end{aligned}
\end{equation}
for some functions $T^a_\pm,W_i,x_j$ on $\cG$. 
We omit long expressions for other curvature functions as they are not relevant for this work.

The harmonic invariants from Proposition \ref{kH} can be represented as two torsion components
(here and below we omit pullback $s^*$ via a section  $s:M\to\cG$ for forms on the structure bundle)
 \begin{equation}\label{eq:harmonic-torsions}
  \begin{gathered}
    \tau_{+}=(T^1_+{\partial_{\omega^3}}+T^2_+\partial_{\omega^4})\otimes (\omega^1\w\omega^2),\quad
    \tau_-= (T^1_-\partial_{\omega^1}+T^2_-\partial_{\omega^2})\otimes (\omega^3\w\omega^4).
  \end{gathered}
 \end{equation}
and the curvature component (this and similar invariants will be treated as tensors on $M$)
 \begin{equation}\label{eq:hamornic-curvature}
\!\!\!\!\! \begin{aligned}
W&= \left(W_0(\omega^1)^4+W_1(\omega^1)^3\omega^2+W_2(\omega^1)^2(\omega^2)^2
+W_3\omega^1(\omega^2)^3+W_4(\omega^2)^4\right)\otimes U \otimes (V_-)^{-2} \\
       &+\left(W_0(\omega^4)^4-W_1(\omega^4)^3\omega^3+W_2(\omega^4)^2(\omega^3)^2-W_3\omega^4(\omega^3)^3+W_4(\omega^4)^4\right)\otimes U \otimes (V_+)^{-2} \\
 \end{aligned}
  \end{equation}
where $U=\omega^0,$ $V_-=\omega^1\w\omega^2,$ $V_+=\omega^3\w\omega^4$. 
As a result,  
 $$
W\in\Gamma\left((S^4L_-^*)\otimes(\Lambda^2L^*_-)^{-2}+(S^4L_+^*)\otimes(\Lambda^2L^*_+)^{-2}\right) \otimes \Delta^\perp
 $$ 
and $\tau_\pm\in\Gamma(\Lambda^2L^*_\mp\otimes L_\pm)$.

Note that for $v=a\,(\p_{\omega^1}+\l\p{\omega^2})+b\,(\l\p{\omega^3}-\p{\omega^4})\in\Pi$ we get 
$W(v,v,v,v)=a^4 W_i\lambda^i+b^4 W_i(-\lambda)^i$ and each summand, in turn, can represent the harmonic
curvature.


The correspondence space $\hat{M}$, as a parabolic geometry $(\mu\colon\cG\to\hat{M},\psi)$ of type 
$(A_3,P_{123})$, is the leaf space of the Pfaffian system $\{\omega^0,\cdots,\omega^5\}$.  
The structure equations imply that the zero-curvature condition $W=0$ is equivalent to 
the Frobenius integrability of the Pfaffian system $\{\omega^0,\omega^1,\omega^4,\omega^5\}$ on $\hat{M}$. 
Denote its 4-dimensional leaf space by $\mathcal{T}$. 
Thus, the zero-curvature condition gives a fibration  $\hat{M}\to\mathcal{T}$ with 2-dimensional fibers. 
    
Equivalently, as was discussed in Section \ref{Sec:2.3}, $\hat{M}$ has a (3,5,6) distribution,  
where the rank 3 distribution has a splitting into a line field and an integrable corank one subdistribution.     
In \cite{Mak2} such (3,5,6) distributions are referred to as \emph{causal structures} on $\mathcal{T}$ 
wherein the fibers of $\hat{M}\to\mathcal{T}$ at each point $x\in\mathcal{T}$ can be locally realized as 
the projectivization of a cone of codimension one in $T_x\mathcal{T}$ whose Gauss map has maximal rank. 

In terms of the Cartan geometry $(\mu\colon\cG\to\hat{M},\psi)$ the rank 3 distribution is given by     
$\hat\Pi\oplus\hat\ell=\mu_*\langle\omega^0,\omega^1,\omega^4\rangle^\perp$, where 
$\hat\Pi:=\mu_*\langle\omega^0,\omega^1,\omega^4,\omega^5\rangle^\perp\subset T\hat{M}$ 
is  integrable and equipped with an indefinite bilinear form. In terms of the structure equations 
\eqref{eq:CartanCurvature} the conformal class of the bilinear form 
$\omega^2\circ\omega^3\in S^2\hat\Pi^*$ is well-defined.
Moreover, the line field $\hat\ell=\mu_*\langle\omega^0,\omega^1,\omega^2,\omega^3,\omega^4\rangle^\perp$ 
is the characteristic direction of $\omega^0$ on $\hat{M}$ i.e. $\hat\ell=\langle v\rangle\subset T\hat{M}$, 
where $\omega^0(v)=0$ and $d\omega^0(v,\cdot)=0$. 
The integral curves of $\hat\ell$ foliate $\hat{M}$ and can be thought of as a generalization 
of the null geodesic spray in conformal pseudo-Riemannian structures to causal structures. On $\hat M$ the conformal class $[s^* h],$ where $h=\omega^0\omega^5-\omega^1\omega^4$ and $s\colon\hat M\to \cG$ is a section, is well-defined. If the fibers of the causal structure $\hat M\to \mathcal T$ are the projective quadric, i.e. $\tau_\pm=0,$ then $[s^* h]$ defines an indefinite conformal structure on $\mathcal T$  whose projectivized null cone bundle coincides with $\hat M.$

The almost CR case is similar, except that instead of having two  components $\tau_\pm$ in \eqref{eq:harmonic-torsions}, 
the torsion has only one irreducible component, so we omit the respective arguments. 

Lastly, we point out that subconformal structures on contact 5-manifolds that we consider in this article are also 
referred to as 5-dimensional Lie contact structures of signature (1,1) for $\delta>0$ and (2,0) for $\delta<0$. 


\section{Proof of the main results}\label{Sec:3}

Let $\E:F=0$ be a PDE system in terms of $u$. Following \cite{CK} we introduce dLp's as follows.

 \begin{definition}\label{DLP}
A \emph{dispersionless pair} is a bundle $\pi:\hat{M}_u\to M_u,$ called the
\emph{correspondence space}, whose fibers are connected curves, together with a rank two
distribution $\hat\Pi\subset T\hat M_u$ such that:
 \begin{itemize}
\item for all $\hat x\in \hat M_u$, $\hat\Pi_{\hat x}\subset T_{\hat x} \hat{M}_u$ depends on a finite jet 
	of $u$ at $x=\pi(\hat x)\in M_u$;
\item $\hat\Pi$ is transverse to the fibers of $\pi$, i.e.\ $\hat\Pi\cap\op{Ker}\pi_*=0$.
\end{itemize}
A \emph{spectral parameter} is a local fiber coordinate $\l=\l(\hat x)\colon
\hat M_u\to\RR$ for $\hat M_u\to M_u$.
\end{definition}

 \begin{definition}\label{DIS}
Two dispersionless pairs $\hat\Pi,\hat\Pi'\subset T\hat M_u$ are \emph{$\E$-equivalent}
if $\hat\Pi=\hat\Pi'$ whenever $F(u)=0$.
$\hat\Pi$ is a \emph{dispersionless Lax pair} (\emph{dLp}) for $\E$ if
  for any $\hat\Pi'$ $\E$-equivalent to $\hat\Pi$, the integrability condition
  $[\hat\Pi',\hat\Pi']= \hat\Pi'$ is a nontrivial differential corollary of
  $\E$.
 \end{definition}

To be precise with the notion of a differential corollary and to encompass systems of PDEs,
we introduce some jet formalism, for which we refer to \cite{KV,KL} for further details.

\subsection{Jets, symbols and characteristics}\label{Sec:3.1}

Consider a (vector) bundle $\nu:\V\to M$ of rank $m$ with local fiber coordinates $u=(u^j)$, so that sections
have coordinate expression $u=u(x)$, $x=(x^i)$ ($1\leq i\leq d$, $1\leq j\leq m$).
A $k$-jet of $u$ is an equivalence class of sections by tangency of order $>k$ relation,
and in coordinates it can be  written as $j^ku=(x,u,\p u,\ldots,\p^ku)$, where $\p^lu=(\p_\sigma u^j)$
with the multi-index $\sigma=(i_1,\dots,i_d)$ of length $|\sigma|=\sum_1^di_s=l\leq k$.

This yields the space of $k$-jets $J^k\nu$ and its projective (inverse) limit $J^\infty\nu$.
There are natural projections $\nu_k:J^k\nu\to M$, $k=0,1,\dots,\infty$, and also
$\nu_{k,l}:J^k\nu\to J^l\nu$ for $k>l$. The fibers of $\nu_{k,k-1}$ for $k\ge2$
have a natural affine structure associated with fibers of $S^k T^*M\otimes\V$.
Any section $u:M\to\V$ canonically lifts to the jet section $j^ku$ of $J^k\nu$.

By $f\in C^\infty(J^\infty\nu)$ we mean a function $f$ on $J^k\nu$ for some finite $k$,
which corresponds to a (nonlinear) differential operator of order $k$.
A collection of such functions $F=(F_1,\dots,F_m)$ can be seen as a vector-valued
differential operator $F:J^k\nu\to\W$, where the latter is another (vector) bundle over $M$.

The bundle $J^\infty\nu$ has a canonical flat connection, the \emph{Cartan distribution},
for which the horizontal lift of a vector field $X$ on $M$ is the \emph{total derivative} $D_X$
characterized by $(D_Xf)\circ j^\infty u = X(f\circ j^\infty u)$ for any smooth function $f$ on $J^\infty\nu$.
More generally, any section $X$ of $\nu_\infty^*TM$ has a lift to a vector field $D_X$ on
$J^\infty\nu$, given in local coordinates by $D_X=\sum_i a_iD_i$,
where $X=\sum_i a_i\p_i$, in which  $\p_i=\p_{x^i}$, and $D_i=\p_i+\sum_\alpha u^j_{i\alpha}\p_{u^j_\alpha}$.

Higher order operators $\Box$ in total derivatives, also known as
\emph{$\Cc$-differential operators}, are generated as compositions of the
derivations $D_X$ with coefficients being smooth functions on $J^\infty\nu$.
In local coordinates, $\Box=\sum a_\alpha D_\alpha$, where $a_\alpha\in
C^\infty(J^\infty\nu)$ and $D_\alpha=D_{i_1}\cdots D_{i_j}$ for a multi-index
$\alpha=(i_1,\ldots i_j)$ with entries in $\{1,2,\ldots d\}$.

A PDE of order $k$ is defined as an equation of the form
 \begin{equation}\label{F-ell}
F(j^ku)=0
 \end{equation}
where $F\in C^\infty(J^k\nu,\W)$ is a vector-function.

Let $\I_F$ be the ideal in $C^\infty(J^\infty\nu)$ generated by the pullback of
$F\in C^\infty(J^k\nu)$ and its total derivatives of arbitrary order. Then
the zero-set $\E_\infty\subset J^\infty\nu$ of $\I_F$ is the space of formal solutions
of~\eqref{F-ell}: $u$ is a solution of~\eqref{F-ell} if and only if  $j^\infty u$ is a
section of $\E_\infty$. This embeds $M_u$ to $\E_\infty$.

In this formalism, a \emph{differential corollary} of $\E:F=0$ is a differential ideal
$\I\subset \I_F$; it is \emph{nontrivial} provided that it is not a subset of $\I_{F'}$ for any
$F'$ with the zero-locus being a proper (closed) subset of that for $F$.
Thus, in Definition~\ref{DIS}, the integrability condition for a dLp
$\hat\Pi$ for $\E:F=0$ need not generate $\I_F$: indeed, the freedom to replace
a dLp by an $\E$-equivalent one may change the ideal $\I\subset \I_F$ that its
integrability conditions generate.

For a function $F\in C^\infty(J^k\nu)$ the vertical part of the $1$-form
$dF\in\Omega^1(J^\infty\nu)$ may be viewed in coordinates as a (vector-valued) polynomial
on $\nu_\infty^*T^*M$ given by
 \[
\sum_{j=0}^k F_{(j)}\quad\text{where} \quad F_{(j)} =\sum_{|\alpha|=j}
(\p_{u_\alpha} F) \p_\alpha\quad\text{is a section of}\quad\nu_\infty^*S^jTM\otimes\V^*.
 \]
The top degree term $\z_F=F_{(k)}$, called the (order $k$) \emph{symbol}
of $F$, is independent of coordinates. We assume it is nonvanishing: if it
vanishes, $F$ has order $\leq k-1$ and $\z_F$ has lower degree.

Similarly, for a $\W$-valued vector-function $F$ on $J^k\nu$ the symbol
$\z_F$ is a homogeneous degree $k$ polynomial on $\pi_\infty^*T^*M$ with values
in $\op{Hom}(\V,\W)$. For a PDE system $\E:F=0$ of order $k$ it is not identically zero,
and the \emph{characteristic variety} is defined by~\cite{Sp}
 \[
\op{Char}(\E,u)=\{[\theta]\in\PP(\pi_\infty^* T^*M_u)\,|\, \z_F(\theta)
\text{ is not injective}\}.
 \]
If $\V$ and $\W$ have the same rank $m$, then $\z_F$ is represented by a $m\times m$ matrix $[\z_{ij}]$
of polynomials of order $k$, and the (projective) covector $[\theta]$ is characteristic if and only if
$\z_F(\theta)$ is not bijective, whence
 $$
\op{Char}(\E,u)=\{[\theta]\in\PP(\pi_\infty^* T^*M_u)\,|\, \op{det}[\z_{ij}(\theta)]=0\}.
 $$
Thus, a determined system is defined by the condition $\op{codim}\op{Char}(\E)=1$.

\subsection{Normality condition}\label{Sec:3.2}

In this section we  prove Theorem \ref{Th1}.

In order for $\hat\Pi$ to be a dispersionless Lax pair for an equation $\E:
F=0$, we require that the integrability condition $[\hat\Pi,\hat\Pi]=\hat\Pi$
holds modulo $\E$, i.e.\ when $F=0$, or, to use physics terminology, \emph{on shell}.

  \begin{definition} 
We say that the dispersionless pair $\hat\Pi\subset T\hat M_u$
is \emph{normal} if $[\hat\Pi,\hat\Pi]\subset\pi_*^{-1}\Pi$ \emph{off shell}, i.e.\ without
assuming $F=0$. In other words, $\pi_*[\hat\Pi,\hat\Pi]=\Pi$.
  \end{definition}

Consider local Darboux coordinates $(x,y,p,q,r)$ on $M_u$ so that for
$\omega^0=dr-p\,dx-q\,dy$ we get $\Delta=\op{Ker}(\omega^0)=\langle v_1,v_2,w_1,w_2\rangle$
with $v_1=\p_x+p\p_r$, $v_2=\p_y+q\p_r$, $w_1=\p_p$, $w_2=\p_q$.
Let $\lambda$ be a local coordinate on the fiber of $\pi:\hat{M}_u\to M_u$. 
Then, because a dLp is characteristic \cite{CK}, we get that $\Pi\subset\Delta$ can be chosen
\emph{isotropic} off-shell, and so in an open region of $\hat{M}_u$ generators
of the distribution $\hat\Pi$, as well as Darboux coordinates, can be chosen so that
 \begin{equation}\label{XYhat}
\hat{X}=v_1+a\,w_1+b\,w_2+m\,\p_\l,\quad \hat{Y}=v_2+b\,w_1+c\,w_2+n\,\p_\l,
 \end{equation}
for some coefficients $a,b,c,m,n$ depending on coordinates of $J^\infty\nu$ and $\l$.

Then $\hat\Pi=\langle\hat{X},\hat{Y}\rangle$ is normal if and only if $[\hat{X},\hat{Y}]$
is a multiple of $\p_\l$. In this case the
integrability condition reduces to the vanishing of the $\p_\l$-component
$\hat X(n)-\hat Y(m)$ of the vector field $[\hat X,\hat Y]$ with no $\lambda$-dependency.
The genericity condition we need here is as follows.

 \begin{definition}\label{df:nondegeneracy}
An isotropic $2$-plane congruence $\Pi=\op{Ker}\{\omega^0,\zeta,\theta\}\subset\Delta$
is called \emph{nondegenerate} if 
 \begin{equation}\label{eq:ndg}
\theta\wedge\zeta\wedge\theta_\l\wedge\zeta_\l\wedge\omega^0\neq 0.
\end{equation}
 \end{definition}

This condition depends only on $\Pi$, and not on a choice of generators $\omega^0,\zeta,\theta$ of $\op{Ann}(\Pi)$.
Indeed the nondegeneracy can be expressed, in terms of the generators
$X=v_1+a\,w_1+b\,w_2$, $Y=v_2+b\,w_1+c\,w_2$ of $\Pi$, as
 $$
(d\omega^0\w d\omega^0)(X,Y,X_\l,Y_\l)\neq0,
 $$
or in terms of their coefficients as follows:
 \begin{equation}\label{z2}
a_{\l}c_{\l}-b_{\l}^2\neq 0
 \end{equation}
where $a\lambda=\partial_\lambda a.$
 \begin{lemma} \label{lem4D}
Any nondegenerate isotropic $2$-plane congruence $\Pi$ has a unique normal lift.
 \end{lemma}

 \begin{proof}
If $\hat X$ and $\hat Y$ are given by~\eqref{XYhat}, then equalities
$dx([\hat X,\hat Y])=dy([\hat X,\hat Y])=dr([\hat X,\hat Y])=0$ hold identically, while
$dp([\hat X,\hat Y])=dq([\hat X,\hat Y])=0$ form two linear equations on $m,n$:
 \begin{equation*}
\begin{bmatrix} a_\l & b_\l \\ b_\l & c_\l\end{bmatrix}
\begin{bmatrix}n \\ -m\end{bmatrix} =
\begin{bmatrix} (X(b)\!-\!Y(a)) \\ (X(c)\!-\!Y(b))\end{bmatrix};
 \end{equation*}
these have a unique solution by the nondegeneracy condition~\eqref{z2}.
 \end{proof}

 \begin{proposition}\label{p:normal}
Let $\hat\Pi$ be a dLp such that $\Pi=\pi_*(\hat\Pi)$ is nondegenerate. Then $\hat\Pi$ is $\E$-equivalent
to a normal dLp. Such a dispersionless Lax pair is unique.
 \end{proposition}

 \begin{proof}
The on shell Lax pair condition implies
 \[
dp\circ\pi_*[\hat X,\hat Y]=\Box_1F,\qquad dq\circ\pi_*[\hat X,\hat Y]=\Box_2F
 \]
for some operators $\Box_1,\Box_2$ in total derivatives.  Let us modify
$\tilde X=\hat X+A(F)\p_\l$, $\tilde Y=\hat Y+B(F)\p_\l$, where $A,B$ are
operators in total derivatives.
The new commutation equation modulo $\p_\l$ is
 \begin{gather*}
dp\circ\pi_*[\tilde X,\tilde Y]=(\Box_1 - a_\l B + b_\l A)F\\
dq\circ\pi_*[\tilde X,\tilde Y]=(\Box_2 - b_\l B + c_\l A)F.
 \end{gather*}
Vanishing of these, equivalent to normality, can be achieved by a unique
choice of the operators in total derivatives $A,B$ due to nondegeneracy
condition \eqref{z2}.
 \end{proof}

This finishes the proof of Theorem \ref{Th1}.

\subsection{Zero-curvature condition}\label{Sec:3.3}

As was noted above, due to the characteristic condition \cite{CK}, the 2-plane congruence
is isotropic both with respect to conformal symplectic and subconformal structures.
In other words, for every $x\in M$, $\Pi\in\op{Gr}(2,T_xM)$ as a function of $\lambda$ is
a section of $\mathcal{A}_x\cup\mathcal{B}_x$. However, as we saw in \ref{Sec:2.1},
isotropic planes in $\mathcal{B}_x$ are discrete (two points for $\delta>0$ and empty for $\delta<0$).
Thus, if we postulate essential dependence on $\lambda$, i.e.\ the spectral parameter
is \emph{non-removable}, the congruence has to take values in the bundle of $\alpha$-planes.

 \begin{definition}
A dLp $\hat{\Pi}$ is called \emph{immersed} if the underlying 2-plane congruence consists of $\alpha$-planes
and the map $\lambda\mapsto\Pi$ is an immersion to $\PP^1=\mathcal{A}_x\subset\op{Gr}(2,T_xM)$ for every $x\in M$.
 \end{definition}

 \begin{lemma}\label{LlL}
A dLp $\hat{\Pi}$ is immersed if and only if  it is nondegenerate.
 \end{lemma}

 \begin{proof}
Choosing a frame so that the subconformal structure has form \eqref{gg} and the
conformal symplectic is \eqref{Om-} or \eqref{Om+}, the 2-plane congruence  of $\alpha$-planes
takes form \eqref{gPi}. Then we can choose the symplectic basis
$v_1=\p_{\omega^1}$, $v_2=\p_{\omega^4}$, $w_1=-\p_{\omega^3}$, $w_2=\p_{\omega^2}$.
With respect to these choices one obtains $a=c=0$ and  $b=\l$, so that condition \eqref{z2} holds.
In other words, in the immersed case we can take the coordinate $\lambda$ on $\PP^1$
to be a spectral parameter, which implies nondegeneracy, and the converse
follows from the same computation if we assume, for instance, $b_\l\neq0$.
 \end{proof}

 \begin{proof}[Proof of Theorem \ref{Th2}]
Let us first note that if $F$ has order $k$, then the subconformal structure $(\Delta,c_F)$
has order  $\leq k$ in $u$ and, therefore, is well-defined and, moreover, is nondegenerate for almost any $u,$ which may not necessarily be a solution.  Note that the order is $<k$, for example, if $F$ is quasilinear.
 By Lemma \ref{lem4D} and Proposition \ref{p:normal} the normal lift to the correspondence space
is a first order operator, and hence the standard dLp has order $\leq k+1$ in $u$.

Suppose next that $\hat\Pi\subset T\hat M_u$ is a dLp for $\E$. Then
$\Pi=\pi_*(\hat\Pi)$ is characteristic, and hence it is a nondegenerate
congruence of $\alpha$-planes for generic $u$. Then by Lemma \ref{LlL}
$\Pi$ immerses into $\op{Gr}(2,TM_u)$ and hence, by Theorem \ref{Th1},
$\hat\Pi$ is $\E$-equivalent to a standard dLp over any open subset of $M_u$.
By the results of Section \ref{Sec:2.2} the curvature is zero on $M_u$
for every solution $u$ of $\E$.
Hence the zero-curvature condition is a nontrivial differential corollary of $\E$, as required.

Conversely, suppose that the condition $W=0$ is a nontrivial differential
corollary of $\E$, and let $\hat\pi:\hat M_u\to M_u$ be the bundle of $\alpha$-planes.
Then if $\hat\Pi$ is $\E$-equivalent to a standard dLp on an open subset of $M_u$,
the integrability of $\hat\Pi$ is a differential corollary of $\E$ on that open subset, since this is true for the standard dLp.
If any such $\hat\Pi$ is a differential corollary of a proper subsystem
$\E'$ of $\E$, then the first part of the argument implies that the zero-curvature condition
is also a consequence of $\E'$, contradicting nontriviality.
\end{proof}

\subsection{Master-equation}\label{Sec:3.4}

We restrict to the almost para-CR case $\delta>0$ and work in the framework of Section \ref{Sec:2.1}.
The spectral parameter $\lambda$ is defined up to projective transformation, depending on the base point $x\in M_u$.
By integrability, there are infinitely many  $\alpha$-surfaces, i.e.\ 2-dimensional submanifolds
of $M_u$ whose tangent plane at each point coincides with  $\Pi=\Pi_\l$ for some $\l.$ They can be thought of as projected integral surfaces of $\hat\Pi,$ and hence, there are  4-parameter family  of them.  
Thus, using the freedom of projective reparametrization of $\PP^1$ for $\a$-planes, 
we can arrange a null and Legendrian foliation  corresponding to the value $\lambda=\infty$. In other words,
we can assume the Legendrian $c_\E$-isotropic  distribution $\Pi_\infty$ to be integrable.

Now we straighten this distribution $\Pi_\infty$, i.e.\ choose Darboux coordinates $(x,y,p,q,r)$ such that
a contact form  $\omega^0\in \op{Ann}(\Delta)$ can be expressed as $\omega^0=dr-p\,dx-q\,dy$ and that $\Pi_\infty=\langle\p_p,\p_q\rangle.$ Note that 
straightening a Legendrian foliation is possible by a canonical transformation. Then
 $$
g=a_{11}\,dx\,dp+a_{12}\,dx\,dq+a_{21}\,dy\,dp+a_{22}\,dy\,dq +b_{11}\,dx^2+2\,b_{12}\,dx\,dy+b_{22}\,dy^2.
 $$
Computing the operator $J$, the compatibility conditions between $g$ and $\Omega$, which is equivalent to $J^2=c\cdot\1$ 
for a positive constant $c$ on $\Delta$, are the following:
 \begin{equation}\label{tmp}
a_{12}(a_{11}+a_{22})=0,\ a_{21}(a_{11}+a_{22})=0,\ a_{11}^2=a_{22}^2,\
a_{21}b_{11}-a_{12}b_{22}=(a_{11}-a_{22})b_{12}.
 \end{equation}
together with the normalization $a_{12}a_{21}+a_{22}^2=4$ coming
from the constraint $\det J=1$ required to compute $L_{\pm}=\op{Ker}(J\mp\1)$.
Equations \eqref{tmp} branch as follows:
 \begin{itemize}
\item $a_{11}=a_{22}$, $a_{12}=a_{21}=0$;
\item $a_{11}=-a_{22}$, $a_{12}b_{22}-a_{21}b_{11}=2\,a_{22}b_{12}$.
 \end{itemize}
The first branch has less parameters and can be transformed to a particular case 
of the second branch. 
Henceforth we proceed with the latter. Using the conformal freedom for the metric,
we impose a different conformal normalization $a_{12}=1$ to simplify computations.
We get
 $$
g= b_{11}\,dx^2+2\,b_{12}\,dx\,dy+(a_{21}b_{11} + 2\,a_{22}b_{12})\,dy^2
+\,dx\,dq+a_{21}\,dy\,dp+a_{22}\,(dy\,dq-dx\,dp).
 $$
Introducing a change of dependent variables $u=a_{22}$, $v=\sqrt{a_{22}^2+a_{21}}$, $w=b_{11}$, $z=2b_{12}$,
and using the coframe below, in addition to the contact form $\omega^0$,
 \begin{gather*}
\omega^1 = dx + (u+v)\,dy,\quad
\omega^2 = dq - (u+v)\,dp + w\,dx + (z-w(u+v))\,dy,\\
\omega^3 = dq - (u-v)\,dp + w\,dx + (z-w(u-v))\,dy,\quad
\omega^4 = dx + (u-v)\,dy,
 \end{gather*}
we get the following form for a conformal representatives on $\Delta$:
 $$
2g=\omega^1\omega^3+\omega^2\omega^4,\
2v\,\Omega=\omega^1\wedge\omega^3+\omega^2\wedge\omega^4.
 $$
Then, with the notations $\widetilde{\p_x}=\p_x+p\p_r$ and $\widetilde{\p_y}=\p_y+q\p_r$, we get
 \begin{gather*}
L_-=\langle\p_{\omega^1},\p_{\omega^2}\rangle=
\langle(u-v)\widetilde{\p_x}-\widetilde{\p_y}+w\p_p-(w(u-v)-z)\p_q, \p_p+(u-v)\p_q\rangle,\\
L_+=\langle\p_{\omega^4},\p_{\omega^3}\rangle=
\langle(u+v)\widetilde{\p_x}-\widetilde{\p_y}+w\p_p-(w(u+v)-z)\p_q, \p_p+(u+v)\p_q\rangle.
 \end{gather*}
These define the para-CR structure on $M_u$, whose family of $\alpha$-planes is
$\Pi_\lambda=\langle\p_{\omega^1}+\lambda\p_{\omega^2},\lambda\p_{\omega^3}-\p_{\omega^4}\rangle$.
A linear combination with the change of the parameter $\lambda\mapsto -v\lambda$ yields new generators
 \begin{equation}\label{MEdLp1}
X=\widetilde{\p_x}+\lambda\p_p+(\lambda u-w)\p_q,\ Y=\widetilde{\p_y}+ (\lambda u-w)\p_p+(\lambda(u^2-v^2)-z)\p_q.
 \end{equation}
The corresponding (standard) dLp is
 \begin{equation}\label{MEdLp2}
\hat{X}=X+m\p_\lambda,\quad \hat{Y}=Y+n\p_\lambda, 
 \end{equation}
where, by the normality condition,
 \com{
 \begin{align*}
m =&\, \frac1{v^2}\Bigl( \bigl(v^2u_q -2u\,v\,v_q - 2v\,v_p\bigr)\lambda^2
+ \bigl(2u^2w_q + u\,(p\,u_r-w\,u_q + u_x + 2w_p - z_q) \\
 & - v^2w_q + v\,(2w\,v_q - 2p\,v_r - 2v_x)
 + w\,u_p + z\,u_q - q\,u_r - u_y - z_p\bigr)\lambda \\
 & + u\,(p\,w_r -w\,w_q + w_x)  + w\,(z_q-w_p)
 + q\,w_r - p\,z_r - z\,w_q + w_y - z_x\Bigr), \\
n =&\, \frac1{v^2}\Bigl(v\bigl(2u(v\,u_q -u\,v_q - v_p) + v\,u_p\bigr)\lambda^2
 + \bigl(2u^3w_q + u^2(p\,u_r - w\,u_q  + u_x + 2w_p - z_q)\\
 & + v^2(p\,u_r - w\,u_q + u_x - w_p)
 + 2u\,v(w\,v_q - v\,w_q - p\,v_r  - v_x)
 + u(w\,u_p - q\,u_r \\
 & + z\,u_q - u_y - z_p)\bigr)\lambda
 + u^2(p\,w_r - w\,w_q+ w_x)
+ v^2(w\,w_q - p\,w_r - w_x)\\
 & + u\,(w\,z_q - w\,w_p + q\,w_r - p\,z_r - z\,w_q + w_y - z_x)\Bigr).
 \end{align*}
These formulae are rather complicated, so we use
 }
\begin{equation}\label{MEdLp3}
\begin{aligned}
m=&\, \frac1{v^2}\Bigl( \bigl(u\,X(u) - 2v\,X(v) - Y(u)\bigr)\lambda + u\,X(w) - X(z) + Y(w) \Bigr),\\
n=&\, \frac1{v^2}\Bigl( \bigl((u^2+v^2)\,X(u) - 2u\,v\,X(v) - u\,Y(u)\bigr)\lambda
+(u^2-v^2)\,X(w) - u\,X(z) + u\,Y(w) \Bigr).
\end{aligned}
\end{equation}
Now the curvature, which is the obstruction to integrability, is given by
 $$
W=X(n) - Y(m) + m\,n_\lambda - n\,m_\lambda=
 W_0+W_1\lambda+W_2\lambda^2+W_3\lambda^3+W_4\lambda^4.
 $$
In general $W\in\odot^4\RR^2$ is a polynomial of degree 4 in $\lambda$, but due to straightening of $\Pi_\infty$ 
we get $W_4=0$, so integrability is given by 4 second order PDEs
 \begin{equation}\label{master5D}
W_0=W_1=W_2=W_3=0
 \end{equation}
on 4 functions $u,v,w,z$ of $x,y,p,q,r$.
 \com{
Equations take long form in coordinates, for instance
 \begin{multline*}
W_3= v^3\,u_{q,q} - 2v^2(u\,v_{q,q} + u_qv_q + v_{p,q}) +
 u\,v\,(u\,u_{q,q} + 2u_q^2 + 2v_q^2 + 2u_{p,q}) \\
 + v\,(2u_pu_q + 2v_pv_q +  u_{p,p})
 - 2(u\,v_q + v_p)(u\,u_q + u_p).
 \end{multline*}
 }
Using decompositions
 \begin{align*}
X=&\, X_0+\lambda X_1, & X_0=&\, \p_x+p\p_r-w\p_q, & X_1=&\, \p_p+u\p_q,\\
Y=&\, Y_0+\lambda Y_1, & Y_0=&\, \p_y+q\p_r-w\p_p-z\p_q, & Y_1=&\, u\p_p+(u^2-v^2)\p_q.
 \end{align*}
and  $m=m_0+\lambda m_1+\lambda^2 m_2$, $n=n_0+\lambda n_1+\lambda^2 n_2$, where
 \begin{gather*}
m_0 = v^{-2} \bigl(u\,X_0(w) - X_0(z) + Y_0(w) \bigr),\quad
m_2 = v^{-2} \bigl(u\,X_1(u) - 2v\,X_1(v) - Y_1(u)\bigr),\\
m_1 = v^{-2} \bigl(u\,X_0(u) - 2v\,X_0(v) - Y_0(u) + u\,X_1(w) - X_1(z) + Y_1(w) \bigr),\\
n_0 = v^{-2} \bigl((u^2-v^2)\,X_0(w) - u\,X_0(z) + u\,Y_0(w) \bigr),\\
n_1 = v^{-2} \bigl((u^2+v^2)\,X_0(u) -2u\,v\,X_0(v) -u\,Y_0(u) +(u^2-v^2)\,X_1(w) -u\,X_1(z) +u\,Y_1(w) \bigr),\\
n_2 = v^{-2} \bigl((u^2+v^2)\,X_1(u) - 2u\,v\,X_1(v) - u\,Y_1(u)\bigr),
 \end{gather*}
the components of \eqref{master5D} are given by
 \begin{equation}\label{W0123}
 \begin{gathered}
W_0=X_0(n_0) - Y_0(m_0) + m_0n_1 - n_0m_1,\\
W_1=X_0(n_1)+X_1(n_0) -Y_0(m_1)-Y_1(m_0) +2(m_0n_2 - n_0m_2),\\
W_2=X_0(n_2)+X_1(n_1) -Y_0(m_2)-Y_1(m_1) +m_1n_2 - n_1m_2,\\
W_3=X_1(n_2) - Y_1(m_2).
 \end{gathered}
 \end{equation}

 \begin{proof}[Proof of Theorem \ref{Th3}]
Master equation \eqref{W0123} is a system of 4 second-order PDEs on 4 functions. 
Its symbol given by the $4\times4$ matrix $B$, whose row number $k+1$ ($0\le k\le3$) 
consists of symbols of $W_k$ by variables $u,v,w,z$ in turn. 
(The symbol of $m$-th order differential operator $F$ by $u$ is $\sum\frac{\p F}{\p u_{\a}}\p_\a$, 
where $\a=(\a_1,\dots,\a_5)$ is a multi-index of length $|\a|=m$ and $\p_\a=\p_1^{\a_1}\cdots\p_5^{\a_5}$.)

Up to a nonzero factor,  $\det(B)\propto Q^4$, where $Q=\sum b^{ij}\p_i\p_j$ is a quadric
in the generators $\p_1+p\p_3,\p_2+q\p_3,\p_4,\p_5$ and hence has rank 4.
The matrix $g^{ij}$ of this bilinear form in the given basis is the inverse to the matrix $g_{ij}$
of the subconformal metric $g$ representing $c_\E$. Thus: 
 \begin{itemize}
\item this system of equation recovers the subconformal structure it describes;
\item the symbol is nondegenerate, i.e.\ $\det B\not\equiv0$ as a function on $T^*M$;  
\item the characteristic variety is a (degenerate quadratic) hypersurface.
 \end{itemize}
Hence system \eqref{W0123} is determined, and, assuming analyticity, 
the generality of its local solutions can be read off from the Cauchy data: $4\cdot2=8$ functions of 4 variables.

To determine functional freedom of zero-curvature subconformal structures,
let us compute the equivalence pseudogroup consisting of transformations that leave our normalizations invariant 
and act as a symmetry on the considered PDE system; a priori the quotient by this symmetry can reduce the
na\"ive count of functional parameters.

Consider the contact vector field $X_f$ on $M_u$, given by 1 generating function $f$ of 5 arguments,
and its action of the most general conformal metric $g$, parametrized by 4 functions
$u,v,w,z$. The flow of $X_f$ preserves the family of such metrics (shape preserving transformation,
cf.\ \cite{KS1,KS2}) if and only if
 $$
L_{X_f}g=c\,g+g'\,\op{mod}\omega^0.
 $$
Here $c$ is a scalar function and $g'$ is the metric in the family with parameters $u,v,w,z$ changed to
parameters $u',v',w',z'$. Taking components and eliminating $c,u',v',w',z'$ one obtains five equations, of
which only two are linearly independent:
 $$
f_{y_1,y_1}=(u^2-v^2)f_{y_2,y_2},\  f_{y_1,y_2}=-uf_{y_2,y_2}.
 $$
Concentrating for a moment on dependence on $y_1,y_2$ only,
the characteristic variety of this system
$\{[\theta]=[\theta_1:\theta_2]\in\mathbb{P}^1:\theta_1^2=(u^2-v^2)\theta_2^2,\theta_1\theta_2=-u\theta_2^2\}$
is empty for $v\neq0$ hence it is of finite type. In other words, $f$ depends on functions of 3 arguments only.

In fact, first prolongation of this system of PDEs is complete in third derivatives by $y_1,y_2$, hence
its solution depends on at most 4 functions of 3 arguments; and indeed it depends exactly on 4 such functions:
the compatibility condition for the above second order system on $f$ (so-called Mayer bracket \cite{KL})
is precisely the component $W_3$ of the curvature, which is one of PDEs in $\{W_i=0\}_{i=0}^3$.
Since local solutions of \eqref{W0123} depend on 8 functions of 4 arguments,
the equivalence pseudogroup (shape preserving transformations) cannot change this count.
 \end{proof}

\subsection{One more example}\label{Sec:3.5}

The following integrable PDE is a direct reduction of the 6D equation by 
Ferapontov-Khusnutdinova \cite{FK}, rewritten in the second order form by Sergyeyev \cite{Se}:
 \begin{equation}\label{5FK}
F=u_5u_{13}-u_3u_{15}+u_5u_{24}-u_4u_{25}=0.
 \end{equation}

The symbol of $F$ equals $\z_F=v_1\cdot v_3+v_2\cdot v_4$ for 
$v_1=\p_1$, $v_2=\p_2$,
$v_3=u_5\p_3-u_3\p_5$,
$v_4=u_5\p_4-u_4\p_5$,
and the nonholonomic distribution is
 $$
\Delta=\langle v_1,v_2,v_3,v_4\rangle=\op{Ker}(\omega^0).
 $$
In coordinates $x=x^1,y=x^2,r=u,p=u_1,q=u_2$ the contact form is canonical
 $$
\omega^0=dr-p\,dx-q\,dy,
 $$
while the subconformal structure (mod $\omega^0$) is canonical in other coordinates 
 $$
g=u_5\cdot \z_F^{-1}=dx^1\cdot dx^3+dx^2\cdot dx^4.
 $$
Passing from these to the Darboux coordinates above brings $c_\E=[g_F]$ to the form 
of Theorem \ref{Th3}, which realizes \eqref{5FK} as a reduction of the master equation \eqref{W0123};
we skip the long explicit formulae.

The $\a$-planes $\langle v_3-\l v_2,v_4+\l v_1\rangle$ give rise to the dispersionless Lax pair:
 $$
\hat\Pi=\langle \p_3-\tfrac{u_3}{u_5}\,\p_5-\l\p_2,\p_4-\tfrac{u_4}{u_5}\,\p_5+\l\p_1\rangle.
 $$

\section{Symmetry reductions and twistor correspondences}\label{Sec:4}

The integrability via dLp can be conveniently described by the double fibration
\begin{center}
\begin{tikzcd}
& \hat{M}_u^6 \arrow[ld, "\hat{\Pi}^2" '] \arrow[rd, "\mathbb{P}^1"]  \\
\mathcal{T}_u^4 && M^5_u
\end{tikzcd}
\end{center}
where the top is the correspondence space and the left-bottom is the twistor space,
i.e.\ the leaf space of the distribution $\hat{\Pi}^2$ on $\hat{M}$.
Note that in the parabolic twistorial picture $M_u^5$ corresponds to a Cartan geometry of type $(A_3,P_{13})$, and 
$\hat{M}_u^6$ has a Cartan geometry of type  $(A_3,P_{123})$.  As discussed at the end of Section \ref{Sec:2.4},  the projection to 
$\mathcal{T}_u^4$ gives a Cartan geometry of type $(A_3,P_2)$  only if the causal structure arises from a  pseudo-conformal structure, i.e.  the 2-dimensional fibers of $\hat M_u\to \cT$ are quadratic, which is equivalent to the vanishing of the torsions $\tau_\pm$. However, the vanishing of $\tau_\pm$ implies flatness of the subconformal structure on $M,$ which in turn induces flat  pseudo-conformal structure on 
$\cT_u.$ As a result, in general,  the causal structure on $\cT_u$ which is encoded by the fibration $\hat M_u\to \cT,$  defines a field of cones on $\cT_u$ that are not quadratic almost everywhere. In this section we discuss other twistor approaches and symmetry reductions.

\subsection{Subconformal geometry in 5D from 3D projective structures}\label{Sec:4.1}

Parabolic twistor correspondence \cite{Cp} is represented as double fibration below,
where $M^5$ of type $(A_3,P_{13})$ with $\delta=+1$ is assumed to be of zero curvature: $W=0$.
The quotient to both sides exists only in the flat case:
the quotient by $L_\pm$ gives a projective geometry on a 3D manifold $N_\pm$ only if $\tau_\mp=0$
and the Weyl curvature of the projective structure on $N_\pm$  is generated by the torsion component $\tau_\pm$ 
by Proposition~\ref{kH} and the parabolic theory of correspondence spaces \cite{Cp}.

\begin{center}
\begin{tikzcd}
& A_3/P_{13} \arrow[ld, "L_-" '] \arrow[rd, "L_+"] & \\
A_3/P_1 && A_3/P_3
\end{tikzcd}
\qquad\qquad
\begin{tikzcd}
& M^5 \arrow[ld, dashed,"\text{if }\tau_+=0"'] \arrow[rd, dashed,"\text{if }\tau_-=0"] & \\
N_{-}^3 && N_{+}^3
\end{tikzcd}
\end{center}
Conversely, given a projective structure in 3D, which can be taken to be either   $N_+$ or $N_-,$ as they are \emph{projectively dual} to each other, 
we can lift it to a subconformal structure in 5D as follows.
A choice of affine symmetric connection $\nabla$ on $N$ induces a connection on the bundle $\pi:T^*N\to N$ and hence 
a splitting $T(T^*N)=H\oplus V$ into a pair of Lagrangian 3-planes, where $V=\op{Ker}(\pi_*)$ and $H$ is the lift of $TN$.

For $M=\PP T^*N$,
it also induces a connection $\bar\nabla$ on the bundle $\bar\pi:M\to N$, which depends 
only on the projective class $[\nabla]$  \cite{Tak}, whence a splitting 
$T_{\bar{x}}M=H'_{\bar{x}}\oplus\bar{V}_{\bar{x}}$ for any $\bar{x}\in M$, 
where $\bar{V}_{\bar{x}}=\op{Ker}(d_{\bar{x}}\bar{\pi})$ and $H'_{\bar{x}}$ is the lift of $T_xN$.
Since $\bar{x}=(x,[p])$ for $p\in T^*_xN$,  the latter space contain the subspace $\bar{H}_{\bar{x}}$
corresponding to $\op{Ann}(p)\subset T_xN$. Thus, we get the splitting of the contact structure:
$\Delta_{\bar{x}}=d_{\bar{x}}\bar{\pi}^{-1}(p)=\bar{H}_{\bar{x}}\oplus\bar{V}_{\bar{x}}$
into a pair of Lagrangian 2-planes, giving the $L_\pm$ planes of the required subconformal structure on $M$.

At the level of the PDE system arising from the zero-curvature condition, the additional condition 
for half of the torsion to vanish results in an overdetermined system of PDEs, whose general solution corresponds 
to generic 3-dimensional projective structure. Thus, we get:

 \begin{theorem}
Projective lift provides explicit solutions of the master equation \eqref{W0123}. Moduli of the corresponding
zero-curvature subconformal structures depend on 12 functions of 3 variables.
 \end{theorem}

 \begin{remark}
This twistor reduction can be thought of as a 5-dimensional analogue of the so-called
Dunajski-West construction and its generalization due to Calderbank \cite{DW,C2} wherein self-dual 4-manifolds 
with a null conformal Killing vector are shown to have a foliation by null surfaces containing the conformal Killing
field. Consequently, the 2-dimensional leaf space of such null surfaces is   equipped with a projective structure.
 \end{remark}

In coordinates, the projective connection is given through a representative symmetric affine connection with
Christoffel symbols $\Gamma_{ij}^k$ depending on coordinates $(x^0,x^1,x^2)$ of $N$.
Projective transformations $\Gamma_{ij}^k\mapsto\Gamma_{ij}^k+\frac12(\Upsilon_i\delta_j^k+\Upsilon_j\delta_i^k)$,
depending on arbitrary 1-form $\Upsilon_i$,
leave  the Thomas symbols $\Pi^k_{ij}=\Gamma_{ij}^k-\frac14(\Gamma_{li}^l\delta_j^k+\Gamma_{lj}^l\delta_i^k)$ invariant. Note that $\Pi^k_{kj}=0.$

Alternatively, using the lift of $(A_3,P_1)$ structure to $(A_3,P_{12})$ structure,  a projective 
connection can be considered as the second order ODE system of the type (derivation by $t=x^0$)
 \begin{align*}
\ddot{x}^1=
&\,\, \Pi^0_{11}(\dot{x}^1)^3+2\Pi^0_{12}(\dot{x}^1)^2\dot{x}^2+\Pi^0_{22}\dot{x}^1(\dot{x}^2)^2
	+(2\Pi^0_{01}-\Pi^1_{11})(\dot{x}^1)^2\\
& +2(\Pi^0_{02}-\Pi^1_{12})\dot{x}^1\dot{x}^2-\Pi^1_{22}(\dot{x}^2)^2
	+(\Pi^0_{00}-2\Pi^1_{01})\dot{x}^1-2\Pi^1_{02}\dot{x}^2-\Pi^1_{00},\\
\ddot{x}^2=
&\,\, \Pi^0_{11}(\dot{x}^1)^2\dot{x}^2+2\Pi^0_{12}\dot{x}^1(\dot{x}^2)^2+\Pi^0_{22}(\dot{x}^2)^3
	-\Pi^2_{11}(\dot{x}^1)^2+2(\Pi^0_{01}-\Pi^2_{12})\dot{x}^1\dot{x}^2\\
& +(2\Pi^0_{02}-\Pi^2_{22})(\dot{x}^2)^2-2\Pi^2_{01}\dot{x}^1
	+(\Pi^0_{00}-2\Pi^2_{02})\dot{x}^2-\Pi^2_{00}.
 \end{align*}
Choosing local cordinates $[p]=[-1:p_1:p_2]$ in the fiber $\PP T^*_xN$ we get local coordinates $(t,x^1,x^2,p_1,p_2)$ on $M$,
in which a contact form is expressed as $\omega^0=dt-p_1dx^1-p_2dx^2$ and the contact Legendrian structure $L_\pm$ is given by
2-distributions  
 \begin{equation}
H=\langle\eta_1,\eta_2\rangle,\qquad V=\langle\p_{p_1},\p_{p_2}\rangle,
 \end{equation}
where $V$ is integrable and  
 \begin{multline*}
\eta_i=\p_{x^i}+p_i\p_t-\Pi^k_{00}p_ip_jp_k\p_{p_j}-\sum_{k\neq i}\Pi^k_{0i}p_jp_k\p_{p_j}
-\sum_{k\neq j}\Pi^k_{0j}p_ip_k\p_{p_j}+(\Pi^0_{00}-\Pi^k_{0k})p_ip_j\p_{p_j}\\
-\sum_{k\neq i,j}\Pi^k_{ij}p_k\p_{p_j}+\sum_{j\neq i}(\Pi^0_{0j}-\Pi^i_{ij})p_i\p_{p_j}
+\sum_{j\neq i}(\Pi^0_{0i}-\Pi^j_{ij})p_j\p_{p_j}+(2\Pi^0_{0i}-\Pi^i_{ii})p_i\p_{p_i}+\Pi^0_{ij}\p_{p_j}.
 \end{multline*}
 in which the summation is over repeated indices  $1\leq j,k\leq 2,$ with no summation  over $i=1,2$.
 
Denoting $\omega^1,\dots,\omega^4$ the coframe (mod $\omega^0$) dual to the contact frame
$\eta_1,\eta_2,\p_{p_1},\p_{p_2}$ we obtain formulae \eqref{gg} and \eqref{Om-} for the
subconformal structure $(M,\Delta,g)$. 

Note that the coordinate freedom can constrain 3 out of 15 coefficients, e.g.\ $\Pi^0_{00}=\Pi^0_{01}=\Pi^0_{02}=0$;
the remaining 12 Thomas's coefficients give functional parameter on the moduli space.

\subsection{Symmetry reductions and integrable systems in 3D and 4D}\label{Sec:4.2}


Let us assume that a subconformal contact structure $(M,\Delta,[g])$ has an infinitesimal symmetry $\zeta\in\Gamma(TM)$. 
Since $\Delta$ is contact, $\zeta$ is transverse to $\Delta$ almost everywhere, so localizing in $M$, we can assume 
the leaf space of the integral curves of $\zeta$ to be a 4-manifold $Q$ with the  fibration $q:M\to Q$. 

For $x\in M$ we can identify $\Delta_x\simeq T_{q(x)}Q$, and thus $Q$ gets equipped with conformal
metric $[\tg]=q_*[g]$ and conformal symplectic structure $[\tOmega]=q_*[\Omega]$, as well as 
the endomorphism $\tJ=\tg^{-1}\tOmega=q_*J$ satisfying $\tJ^2=\delta\1$, where $\delta=\pm1$.

A contact form on $M$ can be fixed by the relation 
 \begin{equation}\label{omc}
\omega^0(\zeta)=c\ \text{ for }\ c\in\RR_\times.
 \end{equation} 
This induces a homothety class of closed 2-forms $\tOmega$ on $Q$ by $q^*\tOmega=d\omega^0$. 
As a result, the 4-manifold $Q$ is conformally symplectic and has  a homothety class of \emph{almost $\delta$-K\"ahler} of  
neutral signature, i.e.  almost pseudo-K\"ahler for $\delta<0$ and almost para-K\"ahler for $\delta>0$.
The converse statement is true together with the integrability constraint:

 \begin{theorem}\label{ThmCont}
There is a bijective local correspondence between zero-curvature subconformal contact structures in 5D
with an infinitesimal symmetry and (homothety classes of) self-dual almost (pseudo/para-)K\"ahler structures in 4D.
 \end{theorem}
 
The inverse construction is given by the so-called \emph{contactification}  and the correspondence between zero-curvature and 
self-duality condition follows from the existence of $\alpha$-surfaces. Below we briefly describe the construction. 

The homothety class of an almost para-K\"ahler structure on $Q$ defines a Cartan geometry $(\cK\to Q,\tpsi)$ of type 
$(\RR_\times GL(2,\RR)\ltimes\RR^4,\RR_\times GL(2,\RR))$, 
where $\RR_\times GL(2,\RR)\subset GL(4,\RR)$ was described in Corollary \ref{crGL}.
The two cases $\delta<0$ and $\delta>0$ involve different representations but are similar, so we proceed with the latter.  
The Cartan connection takes the form
 \begin{equation}\label{eq:pKahler-homothey}
  \begin{aligned}    
\tpsi&= \begin{pmatrix}
    \tilde\rho-\tphi_0 & \txi_5 & 0 & 0 \\
    \tomega^5 & \tilde\rho-\tphi_1 & 0 & 0 \\
    \tomega^4 & \tomega^3 & \tilde\rho-\tphi_2 & 0 \\
    \tomega^1 & -\tomega^2 & 0 & \tilde\rho+\tphi_2 \end{pmatrix},\quad
\tilde\rho=\tfrac 14(\tphi_0+\tphi_1).
  \end{aligned}
 \end{equation}

 \begin{remark}
Fixing $c=1$ in \eqref{omc} results in $\tilde\rho=0$ and corresponds to a representative of the homothety class in Theorem \ref{ThmCont}. 
In the treatment of \cite{CaSa} the authors keep the homothety factor so that the structure group $G_0$ of the symmetry reduction remains the same as that of the corresponding parabolic geometry in order to exploit the parabolic theory of Weyl connections. 
 \end{remark}

From now on we reduce the scaling factor $\RR_\times$ in the structure group, as in the remark above, which implies $\tilde\phi_0=-\tilde\phi_1$ in \eqref{eq:pKahler-homothey}.
The Cartan connection and curvature are given by:
 \begin{equation}\label{eq:PCS-conn-curv}
  \begin{aligned}    
\tpsi&= \begin{pmatrix}
    \tphi_1 & \txi_5 & 0 & 0 \\
    \tomega^5 & -\tphi_1 & 0 & 0 \\
    \tomega^4 & \tomega^3 & -\tphi_2 & 0 \\
    \tomega^1 & -\tomega^2 & 0 & \tphi_2 \end{pmatrix},\
\tilde\Psi= \begin{pmatrix}
    \tPhi_1 & \tXi_5 & 0 & 0 \\
    \tOmega^5 & -\tPhi_1 & 0 & 0 \\
    \tOmega^4 & \tOmega^3 & -\tPhi_2 & 0 \\
    \tOmega^1 & -\tOmega^2 & 0 & \tPhi_2 \end{pmatrix}
  \end{aligned}
 \end{equation}
where
 \begin{equation}\label{eq:PCS-streq}
  \begin{aligned}
\tOmega^1&=t^1_-\tomega^3\w\tomega^4,\quad 
\tOmega^2=t^2_-\tomega^3\w\tomega^4,\quad 
\tOmega^3=t^1_+\tomega^1\w\tomega^2,\quad 
\tOmega^4=t^2_+\tomega^1\w\tomega^2,\\
\tOmega^5&=-s_0(\tomega^1\w\tomega^3+\tomega^2\w\tomega^4) -2r\tomega^1\w\tomega^4-P_{22}
	\tomega^1\w\tomega^2-P_{33}\tomega^3\w\tomega^4,\\ 
\tPhi_1 &=-2r\tomega^2\w\tomega^4+(r+s_1)(\tomega^1\w\tomega^3+\tomega^2\w\tomega^4)
	+P_{12}\tomega^1\w\tomega^2-P_{34}\tomega^3\w\tomega^4,\\    
\tPhi_2&=C_1\tomega^1\w\tomega^2-(3r+s_1-t^2_-t^2_+)\tomega^1\w\tomega^3 +(t^2_-t^1_++s_2)\tomega^1\w\tomega^4\\
    &\ +(t_-^1t_+^2+s_0)\tomega^2\w\tomega^3 +(3r-s_1-t^1_-t^1_+)\tomega^2\w\tomega^4+C_3\tomega^3\w\tomega^4,\\
\tXi_5&=-s_2(\tomega^1\w\tomega^3+\tomega^2\w\tomega^4)-2r\tomega^2\w\tomega^3+P_{11}
	\tomega^1\w\tomega^2+r\tomega^2\w\tomega^3+P_{44}\tomega^3\w\tomega^4
   \end{aligned}
  \end{equation}
for some functions $t^1_-,t^2_-,t^1_+,t^2_+,s_0,s_1,s_2,r$ and $C_1,C_3,P_{ij}$ on $\cK$.

The fundamental invariants of such Cartan geometries are  $\ttau_\pm,S,R$ given by
 \begin{equation}\label{eq:tildetau}
  \begin{gathered}
\ttau_{+}=(t^1_+\partial_{\tomega^3}+t^2_+\partial_{\tomega^4})\otimes (\tomega^1\w\tomega^2),\quad
\ttau_-= (t^1_-\partial_{\tomega^1}+t^2_-\partial_{\tomega^2})\otimes (\tomega^3\w\tomega^4),\\
\!\!S=(s_0(\omega^2)^2+2s_1\omega^2\omega^1-s_2(\omega^1)^2)\otimes (\tomega^1\w\tomega^2)^{1/2}\otimes 
(\tomega^3\w\tomega^4)^{-1/2},\quad
R=r\tOmega.
  \end{gathered}
 \end{equation}
We also have a closed 2-form
 \[
\tOmega=\tomega^1\w\tomega^3+\tomega^2\w\tomega^4.
 \]
The conformal structure of the  homothetic representative  is
 \begin{equation}\label{ttg}
\tg_0=\tomega^1\tomega^3+\tomega^2\tomega^4
 \end{equation}
for which the Schouten tensor in the coframe $(\tomega^1,\cdots,\tomega^4)$ is given by 
the quantities $P_{ij}$ in \eqref{eq:PCS-streq}, where 
 \begin{equation*}\label{eq:Rho-Tensor}
s_0=P_{23},\quad s_1=\half(P_{13}-P_{24}),\quad s_2=P_{14},\quad r=\tfrac{1}{2}(P_{13}+P_{24}).
 \end{equation*}
The remaining entries of  the Schouten tensor are derived from the torsions $\ttau_\pm$ as follows:
   \begin{equation}
     \label{eq:Rest-Of-Schouten}
     \begin{gathered}
       P_{11}=-\partial_{\tomega^4}t_+^1,\quad P_{22}=\partial_{\tomega^3}t_+^2,\quad
       P_{21}=\half(\partial_{\tomega^3}t_+^1-\partial_{\tomega^4}t_+^2),\quad
       C_1= \half(\partial_{\tomega^3}t_+^1+\partial_{\tomega^4}t_+^2),\\
       P_{33}=-\partial_{\tomega^2}t_-^1,\quad P_{44}=\partial_{\tomega^1}t_-^2,\quad
       P_{43}=\half(\partial_{\tomega^1}t_-^1-\partial_{\tomega^2}t_-^2),\quad
       C_3= -\half(\partial_{\tomega^1}t_-^1+\partial_{\tomega^2}t_-^2).
     \end{gathered}
        \end{equation}
Note that the fundamental invariant $S$ in \eqref{eq:tildetau} is a $\mathrm{GL}(2,\RR)$-component part 
of the Ricci curvature of the almost para-K\"ahler metric $\tg_0$ in \eqref{ttg}, and the scalar curvature is $24r$.
The anti-self-dual Weyl curvature of $\tg_0$ is zero; the self-dual part is represented by the quartic
  \begin{equation}\label{eq:Self-dual-Weyl-Curvature}
(t^2_{-;3}-t^1_{-;4})\lambda^4+4C_3\lambda^3+6(t^1_-t^1_++t^2_-t^2_+-2r)\lambda^2+4C_1\lambda +t^2_{+;1}-t^1_{+;2}
  \end{equation}
on the bundle of $\a$-planes \eqref{gPi} with parameter $\lambda\in\RR\cup\{\infty\}$, 
where $t^k_{\ve;l}=\p_{\tilde\omega^l}t^k_{\ve}$.

To contactify, take a local nonvanishing primitive 1-form $\tomega^0$ of the closed representative  $\tOmega\in[\tOmega]$, i.e.\ $d\tomega^0=\tOmega$. Defining $M=Q\times \RR$ with projection $q: M\to Q$, 
the 1-form $\beta=d t+q^*\tomega^0$ satisfies $d\beta=q^*\tOmega$, whence $\Delta:=\op{Ker}\beta\subset TM$ is a contact distribution. Moreover, by the construction the splitting $TQ=\tL_-\oplus \tL_+$ induces a splitting on  $\Delta$ with the appropriate compatibility conditions for the induced (para-)complex structure. 
This shows that $M$ is equipped with a compatible subconformal contact structure with infinitesimal symmetry $\p_t$. 

To relate the Cartan connections $\tpsi$ and $\psi,$  let us denote by $\hat q\colon\hat\cK\to M$ the pull-back bundle 
$q^*\cK$ of the principal bundle $\cK\to Q$ using the projection $q: M\to Q$.
Let $\omega^0$ denote the lift of the primitive 1-form $\beta$ to $\hat\cK$ by the scaling action.  
The symmetry reduction gives an inclusion $\iota:\hat\cK\to\cG$, where $(\cG\to M,\psi)$ is the Cartan geometry for the corresponding subconformal contact 5-manifold.  Using the expressions \eqref{eq:PCS-conn-curv} and \eqref{eq:A3P123}, one has
 \begin{equation}\label{eq:contactify-coframe}
  \begin{gathered}
\omega^a=\tomega^a\ (0\leq a\leq 4),\quad \omega^5=\tomega^5+\half s_0\omega^0,\quad 
\phi_1=\tphi_1-\half s_1\omega^0,\\
    \phi_2=\tphi_2+(\tfrac 16t^1_-t^1_++\tfrac 16t^2_-t^2_+-r)\omega^0,\quad \xi_5=\txi_5+\half s_2\omega^0.
  \end{gathered}
\end{equation}
wherein we have  suppressed $\iota^*$ on the left hand side and $\hat q^*$ on the right hand side.  

Using the relations above one  can find the Schouten tensor of the reduced subconformal geometry. More precisely, defining the Schouten tensor as $\xi_i=Q_{ij}\omega^j,$  one obtains that the symmetric part $Q_{(ij)}$, $1\leq i,j\leq 4$, is given by
 \[
\begin{gathered}
Q_{(11)}=\half P_{11},\quad Q_{(12)}=\half P_{21},\quad Q_{(22)}=\half P_{22},\quad Q_{(33)}=\half P_{33},\quad 
Q_{(34)}=\half P_{43},\quad Q_{(44)}=\half P_{44}\\
Q_{(13)}= r+\half s_1+\tfrac{1}{24}(t_-^1t_+^1-5t_-^2t_+^2),\quad Q_{(24)}=r-\half s_1+\tfrac{1}{24}(t_-^2t_+^2-5t_-^1t_+^1)\\
Q_{(23)}=\half s_0+\tfrac 14 t_-^1t_+^2,\quad Q_{(14)}=\half s_2+\tfrac 14 t_-^2t_+^1
   \end{gathered}
 \]
and the only non-zero entries of the skew-symmetric part $Q_{[ij]}$ for $1\leq i,j\leq 4$ are
 \[
Q_{[12]}=\tfrac{1}{4}C_1,\quad Q_{[43]}=\tfrac 14C_3.
 \]
The expressions of the entries $Q_{i0}$ and $Q_{0i}$ involved first and second derivative of the fundamental invariants of 
the almost para-K\"ahler structure. Since their expressions are  long and will not be important for us, we will not provide them. 
Using  relations \eqref{eq:contactify-coframe}, one can immediately relate  invariants \eqref{eq:harmonic-torsions} 
to \eqref{eq:tildetau} by
 \[
\tau_\pm= q^*\ttau_\pm.
 \]
An interpretation for the vanishing of $S$ was given in \cite{MS}, using the natural lift $\hat\zeta$
of the infinitesimal symmetry $\zeta$ to its correspondence space  $\hat M_u$ which would be its corresponding causal structure as discussed at the end of Section \ref{Sec:2.4}. It was shown that the infinitesimal symmetry $\hat\zeta$ is \emph{null} with respect to the canonical symmetric bilinear form 
$h=\omega^0\omega^5-\omega^1\omega^4$ of the causal structure in terms of \eqref{eq:A3P123} if and only if the induced structure on $Q$ satisfies $S=0$.

On the level of differential equations and their canonical subconformal structure on solutions,
vanishing of the fundamental invariants is a constraint, not changing the integrability. 
For instance, if in addition to the symmetry on $M_u$ we impose vanishing of the
Ricci curvature of $\tg_0$, namely $P_{ij}=0$ for all $0\leq i,j\leq 4$, then we get
the Pleba\'nski equation (first or second depending on coordinization) describing self-dual gravity \cite{Pb,DFK}.
On the other hand, fewer constraints yields reductions to other geometries as will be discussed in the next section.

Symmetry reduction of integrable PDEs in 5D clearly gives an integrable PDE in 4D, for instance by the results of \cite{FK,CK}.
Imposing more symmetries on $M_u$, we get lower-dimensional reductions, as in \cite{C}. In particular, 
reduction by a 2-dimensional Abelian symmetry group whose intersection with the contact distribution is not null with respect to the indefinite metric on $\Delta$ yields integrable background geometry in 3D, which is Einstein-Weyl. 
Split as two succesive 1-dimensional symmetry reductions, this yields the Jones-Tod correspondence \cite{JT}. 
Again, vanishing of extra invariants leads to further reductions, e.g.\ almost para-K\"ahler 4-manifolds with $\ttau_-=0$ 
give symmetry reduction of 3-dimensional projective structures, which can be viewed as an analogue of the Dunajski-West construction \cite{DW}.
(The local generality of such almost para-K\"ahler structures is 12 functions of 2 variables.)


\subsection{Further reductions: nested Lax sequences}\label{Sec:4.3}

In this section let us restrict to the case $\delta>0$.
We consider a subconformal structure $(M,\Delta,[g])$ with infinitesimal symmetry $\zeta$ and its symmetry reduction, 
namely a homothety class almost para-K\"ahler 4-manifold $(Q,[\tilde g],\tilde J)$, 
with various overdeterminations of the zero-curvature condition. 

We interpret such vanishing conditions as the extendability of the Lax pair to a Lax triple, Lax quadruple, etc.
Since a differential subsystem of an integrable system is an integrable system itself, 
we obtain overdetermined systems that are integrable via Lax distributions of higher rank.
The candidates for these distributions are as follows. 

Consider at first the reduced space $Q$ and its correspondence space $\hat{Q}$, i.e.\ the bundle of $\alpha$-planes 
and the projection $\pi_Q:\hat{Q}\to Q$ with $\PP^1$-fibers. The $\alpha$-planes $\langle\tilde{l}_-,\tilde{l}_+\rangle$,
where $\tilde{l}_\pm\subset\tilde{L}_\pm$, $\tilde{l}_-\perp\tilde{l}_+$, lift to 
the rank 2 distribution $\tilde{\Pi}^2\subset T\hat{Q}$  defined using the Levi-Civita connection
of $\tilde{g}$; this lift is independent of the representative metric in the homothety/conformal class.
Similarly, the 3-planes $\langle\tilde{l}_-\rangle\oplus\tilde{L}_+$ parametrized by $\lambda\in\PP^1$
lift to the rank 3 distribution $\tilde{\Pi}^3\subset T\hat{Q}$. Finally the tangent bundle $TQ$ 
lifts via the Levi-Civita connection to the rank 4 distribution $\tilde{\Pi}^4\subset T\hat{Q}$. 

 \begin{proposition}
Frobenius integrability of $\tilde{\Pi}^2$ is equivalent to the self-duality of $\tilde{g}$. Assuming self-duality,
$\tilde{\Pi}^3$ is integrable if and only if  we have $\tau_-=0$ and $S=0$. 
Assuming the integrability of $\tilde\Pi^2$ and $\tilde\Pi^3,$ then
$\tilde{\Pi}^4$ is integrable if and only if the  Ricci curvature vanishes.
 \end{proposition}

 \begin{proof}
In terms of Cartan connection \eqref{eq:PCS-conn-curv} on the principal bundle $\pi_{\hat Q}\colon\cK\to \hat Q$ we have
$\tilde{\Pi}^2=\pi_{\hat Q *}\langle\tomega^1,\tomega^4,\tomega^5\rangle^\perp$,
$\tilde{\Pi}^3=\pi_{\hat Q *}\langle\tomega^1,\tomega^5\rangle^\perp$ and
$\tilde{\Pi}^4=\pi_{\hat Q *}\langle\tomega^5\rangle^\perp$. 
The first claim on the Frobenius condition for $\tilde{\Pi}^2$ is well-known \cite{Pe}.
Assuming self-duality, the conditions for this Lax pair to be extendable to the Lax triple are straightforward: 
using \eqref{eq:PCS-streq} and \eqref{eq:Rest-Of-Schouten}, the  Frobenius integrability 
of $\tilde{\Pi}^3$ is  equivalent to $\tau_-=0$ and $S=0$. 
Similarly, assuming  $\tau_-=S=0$, the condition $\op{Ric}_{\tilde{g}}=0$ is necessary and sufficient  
for the rank 4 distribution $\tilde{\Pi}^4$ to be a Lax quadruple. 
 \end{proof}

 \begin{remark}\label{rmk:Petrov-type-Q}
Note that by relation \eqref{eq:Rest-Of-Schouten},  when $\tau_-=0$ then the binary quartic \eqref{eq:Self-dual-Weyl-Curvature} has a repeated root of multiplicity at least 2 at $\lambda=\infty$. The condition $\tau_-=0$ and $\op{Ric}_{\tilde{g}}=0$ implies that  \eqref{eq:Self-dual-Weyl-Curvature} has a repeated root of multiplicity at least 3. 
 \end{remark}
Now consider the subconformal structure with the projection $q:M\to Q$ and the corresponding projection
of the correspondence spaces $\hat{q}:\hat{M}\to\hat{Q}$. We define $\hat{\Pi}^{i+1}=\hat{q}_*^{-1}\tilde{\Pi}^i$
for $i=2,3,4$ which together with the rank 2 distribution $\hat{\Pi}^2$ give a flag of suspaces of $T\hat{M}$.
This can be also defined as a lift of subspaces of $TM$ via Weyl connection. Indeed the symmetry $\zeta$ determines a 
reduction of the Cartan structure algebra $\mathfrak{sl}(4,\RR)$ to the opposite parabolic 
$\fp_{13}^\text{op}=\g_{-2}\oplus\g_{-1}\oplus\g_0$, where $\g_0=\mathfrak{gl}(2,\RR)$,
$\g_{-1}=\RR^4$ and $\g_{-2}=\RR$, the latter generated by $\zeta$. This reduction is given by a Weyl structure.

 \begin{proposition} 
Given a zero-curvature subconformal structure  $(M,\Delta,g)$
the rank 3 distribution $\hat{\Pi}^3 \subset T\hat M$ is integrable if and only if $S=0$ holds for the induced structure on $Q$. 
Assuming that $\hat\Pi^3$ is integrable, the integrability of $\hat{\Pi}^4\subset T\hat M$ is equivalent to $\tau_-=0$ and $S=0$ on $Q$. 
Finally, assuming $\hat\Pi^3$ and $\hat\Pi^4$ are integrable,  
the integrability of $\hat{\Pi}^5\subset T\hat M$ is equivalent to the vanishing of $\tau_-$ and $\op{Ric}_{\tilde g}$.
\end{proposition}

 \begin{proof}
Infinitesimal symmetry yields a reduction of the structure bundle given by $\iota:\hat\cK\to\cG$, as discussed before. 
Using the pull-back of the Cartan connection \eqref{eq:A3P123} to $\cK$  and the projection $\pi_{\hat\cK}:\cK\to\hat M$ 
one has
$\hat{\Pi}^2=\pi_{\hat\cK *}\langle\omega^0,\omega^1,\omega^4,\omega^5\rangle^\perp$,
$\hat{\Pi}^3=\pi_{\hat\cK *}\langle\omega^1,\omega^4,\omega^5\rangle^\perp$,
$\hat{\Pi}^4=\pi_{\hat\cK *}\langle\omega^1,\omega^5\rangle^\perp$ and
$\hat{\Pi}^5=\pi_{\hat\cK *}\langle\omega^5\rangle^\perp$. 
Now the proof is a straightforward inspection of Frobenius integrability conditions using the structure equations \eqref{eq:PCS-streq} and the relations \eqref{eq:contactify-coframe} on $\cK$.
 \end{proof}
 
In the Proposition above,  the Frobenius condition for $\hat{\Pi}^3$ given by $S=0$, has no analogs on $\hat{Q}$. 
This provides another interpretation of the vanishing of $S$.

 \begin{remark}
Our hierarcy of reductions, when a Lax pair is extended to a Lax triple and up to a Lax quintuple, are nested:
the smaller equation (larger Lax distribution) is defined when its larger counter-part is integrable.
It turns out that given a zero-curvature subconformal structure with an infinitesimal symmetry and integrable  
$\hat\Pi^i$ for $i=2,3,4,5$, its Cartan  Holonomy is reduced to $\mathfrak p_{13}^\text{op}\subset \mathfrak{sl}(4,\RR)$ via a reduction $\iota\colon\hat\cK\to\cG$ if and only if $C_1=0$, i.e.\ $\ttau_-=0$ and $\op{Ric}_{\tilde g}=0$ and the binary quartic \eqref{eq:Self-dual-Weyl-Curvature} has a repeated root of multiplicity  4 or is zero on $Q$.
 \end{remark}

\subsection{2-nondegenerate CR structures on the  twistor bundle}\label{Sec:4.4}

For our purposes in this section, involving  (para-)CR structures,  define  $\Bbbk_\ve:=\RR[\sqrt{\ve}]$ for $\ve=\pm1$ (also written as 
$\ve=\pm$), i.e.\ $\Bbbk_{-}=\CC$ and $\Bbbk_{+}=\RR$.

 \begin{definition}\label{def:2-nondegenerate-cr}
An almost $\ve$-CR structure of hypersurface type on a manifold $N$ consists of 
a corank one distribution $D\subset TN$ equipped with a field of compatible $\ve$-complex structures
 \begin{equation}\label{JJJ}
\cJ_\ve:D\to D,\quad  \cJ_\ve^2=\ve\1, \quad \op{Tr}(\cJ_\ve)=0,\quad
\mathcal{L}(\cJ_\ve X,\cJ_\ve Y)=-\ve\mathcal{L}(X,Y),
 \end{equation}
where $\mathcal{L}:\Lambda^2D\to TN/D$ is the Levi bracket of $D$ defined by
 \[
\mathcal{L}(X,Y)=[\tilde{X},\tilde{Y}]_x\,\op{mod}D_x\ \text{ for }\ 
X,Y\in D_x,\  \tilde{X},\tilde{Y}\in\Gamma(D), \ \tilde{X}_x=X,\ \tilde{Y}_x=Y. 
\]
An $\ve$-CR structure is characterized by the  condition that $D'$ and $D''$ are Frobenius integrable where $D',D''$ are eigenspaces of $\cJ_\ve$ corresponding to $\sqrt{\ve}$ and $-\sqrt{\ve}$. 

In the Levi degenerate case, denote  the kernel of $\mathcal{L}$ by $K\subset D$,
which is equipped with the splitting $K'\oplus K''=K\otimes\Bbbk_\ve$ (with obvious notations) and 
the higher Levi bracket $\mathcal{L}'_2:K'\otimes D''\to D'/K'$  defined by
 \[
\mathcal{L}'_2(X,Y)=[\tilde{X},\tilde{Y}]_x\,\op{mod}(K'_x\oplus D''_x)\ \text{ for }\ 
X\in K'_x, \ Y\in D''_x,
 \]
where as above $\tilde{X}\in\Gamma(K')$,  $\tilde{Y}\in\Gamma(D'')$,  $\tilde{X}_x=X$,  $\tilde{Y}_x=Y$;
its conjugate $\mathcal{L}''_2:K''\otimes D'\to D''/K''$ is defined similarly.
An $\ve$-CR structure is called 2-nondegenerate if $\mathcal{L}'_2(X,D'')=0$ implies $X=0$ 
and similarly for the conjugate $\mathcal{L}''_2$.   
 \end{definition}

 \begin{remark}
Note that  when the Levi form is nondegenerate the trace condition in \eqref{JJJ}, which is automatic for $\ve=-1$, 
follows also for $\ve=+1$. This condition implies, in turn,  
that the eigenspaces $D',D''$ of $\cJ_\ve$ corresponding to $\sqrt{\ve}$ and $-\sqrt{\ve}$ have equal dimensions,  which gives  $D'\oplus D''=D\otimes\Bbbk_\ve$,  and hence $\dim N$ is odd.

For CR-structures ($\ve=-1$) the conjugate conditions follow automatically (they are required for para-CR case
$\ve=+1$).  Such structures are a subclass of $k$-nondegenerate CR-structures (for some finite $k$), which 
in the analytic case enjoy the following properties:
 \begin{itemize}
\item They cannot be CR-straightened,   i.e. they are  not CR-equivalent to a product \cite{Fr};
\item Their algebra of infinitesimal holomorphic symmetries is finite-dimensional \cite{BER}.
 \end{itemize}
 \end{remark}

Recall that a compatible subconformal structure on a contact 5-manifold possesses a naturally associated 
Cartan bundle $(\mathcal{G}\to M,\psi)$ with parabolic structure group $P_{13}\cong G_0\ltimes P_+$, 
where $G_0\cong (\RR_\times)\times SL(2,\RR)\times(\RR_\times)$ is the reductive part and $P_+$ is the nilradical;
note that $G_0$ action on $P_+$ factors through $R_\times GL(2,\RR)$. 
Via the $GL(2,\RR)$ action, the contact structure $\Delta$ can be equipped with an almost $\ve$-complex structure 
using $J_\ve\in GL(2,\RR)$ given by
 \begin{equation}\label{eq:fJ-+}
J_{-} = \begin{pmatrix} 0 & 1 \\ -1 & 0 \end{pmatrix},\quad
J_{+}= \begin{pmatrix} 0 & 1 \\ 1 & 0 \end{pmatrix},
 \end{equation}
which, as we will show,  gives rise to a 2-nondegenerate $\ve$-CR structure.  

The stabilizer of $J_\ve$ in $P_{13}$ is equal to 
$H^\ve=\bigl((\RR_\times)\times H^\ve_0\times(\RR_\times)\bigr)\ltimes P_+$,  where 
$H_0^\varepsilon$ is $SO(2), SO(1,1)\subset SL(2,\RR)$  for $\varepsilon=-1,1,$ respectively. 
Define the $\varepsilon$-twistor bundle as
 \[
\hat{\cT}_\ve=\mathcal{G}\slash H^\ve=\mathcal{G}\times_P(P\slash H^\ve).
 \]
This bundle can be identified as the bundle of all $g$-compatible $\ve$-complex structures, cf. \eqref{JJJ}, i.e.\ 
endomorphisms on $\Delta$ satisfying:
 \begin{equation}\label{gJe}
J_\ve:\Delta\to\Delta,\quad  J_\ve^2=\ve\1, \quad \op{Tr}(J_\ve)=0,\quad g(J_\ve X,J_\ve Y)=-\ve g(X,Y).
 \end{equation}
Thus, $\nu_\ve:\hat{\cT}_\ve\to M$ is a fiber bundle with the fibers $SL(2,\RR)\slash H_0^\ve$. 

 \begin{lemma}\label{lmT}
Any fiber $SL(2,\RR)\slash H_0^\varepsilon$ is diffeomorphic to the disk $\mathbb{D}^2$ when $\ve=-1$
and to the cylinder $\mathbb{S}^1\times\RR^1$ when $\ve=+1$.
 \end{lemma}

 \begin{proof}
This is obvious from the adjoint action of $SL(2,\RR)$ preserving the Killing form (of Lorentzian signature):
in the Minkowski coordinates $\RR^{1,2}(t,x,y)$, $ds^2=dt^2-dx^2-dy^2$,  the models of 
Lobachevski and de Sitter planes are given by $t^2-x^2-y^2=\pm1$ with the stabilizers of points
being conjugate to $SO(2)$ and $SO(1,1),$ respectively.

It is instructive to note that this isomorphism reflects, actually, the twistor picture:
 \begin{align*}
& Z_{-}(\RR^{2,2})=SO(2,2)/U(2)=SL(2,\RR)/SO(2)\simeq\mathbb{D}^2,\\
& Z_{+}(\RR^{2,2})=SO(2,2)/U(1,1)=SL(2,\RR)/SO(1,1)\simeq\mathbb{S}^1\times\RR^1,
 \end{align*}
which are, respectively,  the spaces of orthogonal complex/product structures in four dimensional space of split signature,
reflecting \eqref{gJe}.
 \end{proof}

Note that the $\ve$-twistor bundle $\hat{\cT}_\ve$ has a codimension 1 distribution, namely
the preimage of the contact distribution $D=d\nu_\ve^{-1}(\Delta)$.  Its space of Cauchy characteristics, or
the kernel of the Levi form, is $K=\op{Ker}(d\nu_\ve)$.
Similar to the classical twistor theory,  $D$ has the induced almost complex/product structure $\cJ_\ve$,
so $(\hat{\cT}_\ve,D,\cJ_\ve)$ is an almost $\ve$-CR manifold.
 
The following is a nonholomorphic higher-dimensional analog of the classical Atiyah-Hitchin-Singer version 
\cite{AHS} of  Penrose's nonlinear graviton construction.

 \begin{theorem}\label{thm:2-nondegenerate-cr-2}
A compatible subconformal structure $(M,\Delta,[g])$ in 5D has zero curvature if and only if
the corresponding almost $\ve$-CR manifold $(\hat{\cT}_\ve,D,\cJ_\ve)$ in 7D is integrable, which makes it
a 2-nondegenerate $\ve$-CR manifold.
 \end{theorem}

Note that $\ve=\pm1$ here is independent of the choice of the invariant $\op{sgn}\delta=\pm1$ of $(M,\Delta,[g])$. 
Furthermore, if the result holds for either $\ve=+1$ or $-1$ then it is true for both $\ve=\pm1.$

\begin{proof}
Note that we have $T\hat{\cT}_\ve=\mathcal{G}\times_{H^\ve}(\g/\h^\ve)$, where $\h^\ve=\op{Lie}(H^\ve)$.
In the $\fp_{13}$ (contact) grading of $\g$, the distribution $D$ corresponds to $\g_{-1}\oplus(\g_0/\h^\ve)$
and $K$ to $\g_0/\h^\ve$.

For $\ve=-1$, in terms of matrix expression \eqref{eq:A3P123}, 
the 1-forms $(\omega^0,\cdots,\omega^4,\omega^5+\xi_5,\phi_0-\phi_1)$ give a coframe on $\hat{\cT}_-$, and
$D=\op{Ker}(\nu_-^*\omega^0)\subset T\hat{\cT}_-$. 
Using the action of $GL(2,\RR)$ and the almost complex structure $J_{-}$ as in \eqref{eq:fJ-+}, 
it follows that the holomorphic $(1,0)$-type distribution in Definition \ref{def:2-nondegenerate-cr} is given by
$D'=(\nu_-)_*\op{Ker}\{\omega^0,\bar\zeta^1,\bar\zeta^2,\bar\zeta^3\}\subset D^\CC$
and $D''=\overline{D'}$, where 
  \[ 
\zeta^1=\omega^1+i\omega^2,\qquad 
\zeta^2=\omega^3+i\omega^4,\qquad 
\zeta^3=\phi_0-\phi_1+i(\omega^5+\xi_5).
 \]
A straightforward inspection of the structure equation \eqref{eq:CartanCurvature} shows that $D'$ is Frobenius integrable 
if and only if $W$ vanishes. Thus, $W=0$ implies CR-integrability of $\cJ_-$.  Moreover one obtains
 \begin{equation}\label{eq:contact-1form}
d\omega^0\equiv \tfrac{1}{2}(\zeta^1\wedge\bar\zeta^2-\zeta^2\wedge\bar\zeta^1)\,\op{mod}\{\omega^0\}.
 \end{equation}
As a result, the Levi bracket is degenerate along $K=K'\oplus K''$, where $K'=\langle\p_{\zeta^3}\rangle$. 
Lastly, the 2-nondegeneracy of $(D,\cJ_-)$ follows from  the symbol algebra of $(\mathfrak{sl}(4,\RR),\fp_{13})$ since
 \begin{equation}\label{eq:2nondegeneracy}
d\overline\zeta^1\equiv\half\zeta^1\w\overline\zeta^3,\quad 
d\overline\zeta^2\equiv-\half\zeta^2\w\overline\zeta^3\quad
\op{mod} \{\omega^0,\overline\zeta^1,\overline\zeta^2\}.
 \end{equation}

For $\varepsilon=+1$, one can proceed similarly. In terms of matrix expression \eqref{eq:A3P123}, 
the 1-forms $(\omega^0,\cdots,\omega^4,\omega^5-\xi_5,\phi_0-\phi_1)$ give a coframe on $\hat{\cT}_+$ and
$D=\op{Ker}(\nu_+^*\omega^0)\subset T\hat{\cT}_+$.  Using the $GL(2,\RR)$ action 
and the expression of $J_{+}$ as in \eqref{eq:fJ-+}, the corresponding para-holomorphic distributions are 
$D'=(\nu_+)_*\op{Ker}\{\omega^0,\bar\zeta^1,\bar\zeta^2,\bar\zeta^3\}\subset D$,
$D''=(\nu_+)_*\op{Ker}\{\omega^0,\zeta^1,\zeta^2,\zeta^3\}\subset D$, where 
  \begin{align*}
\hskip30pt\zeta^1&=\omega^1-\omega^2, &\hskip-20pt
\zeta^2&=\omega^3+\omega^4,&\hskip-20pt
\zeta^3&=\phi_0-\phi_1+(\omega^5-\xi_5),\\
\hskip30pt\bar\zeta^1&=\omega^1+\omega^2, &\hskip-20pt
\bar\zeta^2&=\omega^3-\omega^4,&\hskip-20pt
\bar\zeta^3&=\phi_0-\phi_1-(\omega^5-\xi_5),
 \end{align*}
with respect to which relation \eqref{eq:contact-1form} remains valid. 
It is again a matter of straightforward inspection of the structure equations \eqref{eq:CartanCurvature} 
to show that distributions $D'$ and $D''$ are Frobenius integrable  if and only if $W$ is zero. 
Lastly,  the 2-nondegeneracy of the  para-CR structure follow from the symbol algebra since 
\eqref{eq:2nondegeneracy} remains valid in the para-CR case as well. 
 \end{proof}

The geometric picture presented here has a direct counterpart on the level of 
equations: the dispersionless integrability of PDE $\E$,  which by Theorem \ref{Th2} is given by the zero-curvature condition of the induced subconformal geometry on $M_u$ for a generic background solution $u$ of $\E$,
is also equivalent to (para)CR-integrability of the corresponding structures on $\hat{\cT}_\ve$ with any choice of $\ve=\pm1$.

\begin{remark}
We point out that the 7-dimensional 2-nondegenerate CR structures, associated to zero-curvature parabolic geometries of type $(A_3,P_{13})$ in Theorem \ref{thm:2-nondegenerate-cr-2}, are generalized in \cite{Gr-CR} to other types of parabolic geometries.

Furthermore, the integrability of the complex structure $\cJ_-$ on $D$ can be considered as the odd-dimensional counter-part 
of the same phenomenon on the twistor bundles for para-quaternionic structures,  known as $\beta$-integrable $(2,n)$-Segr\'e structures,  \cite{Met,Zdk}.
\end{remark}

\subsection{Final examples: maximal and submaximal symmetric models}\label{Sec:4.5}

Let us begin with the zero-curvature subconformal structure of maximal symmetry and then we will discuss 
the constructions of this section for submaximal symmetric structures.

\smallskip

{\bf 1.}
The unique zero-curvature subconformal structure of maximal symmetry for $\delta>0$ is $SL(4,\RR)/P_{13}$ and 
that for $\delta<0$ is $SU(2,2)/P_{13}$; the symmetry is 15-dimensional as indicated.

Let us consider the flat para-CR case ($\delta>0$) in details. The model, 
corresponding via the construction from Section \ref{Sec:4.1} to the flat projective structure, is
given by $M=\RR^5(x^1,x^2,u,p_1,p_2)$ and the $L_\pm$ splitting of the contact distribution 
 $$
\Delta=\op{Ker}(\omega^0),\quad \omega^0=du-p_1dx^1-p_2dx^2,
 $$ 
as follows (we again re-denote the subdistributions by $H,V$): 
 $$
\Delta=H\oplus V\ \ \text{ with }\ \ 
H=\langle\p_{x^1}+p_1\p_u,\p_{x^2}+p_2\p_u\rangle\ \text{ and }\ V=\langle\p_{p_1},\p_{p_2}\rangle.
 $$
The subconformal structure is represented by the metric $g=dx^1\,dp_1+dx^2\,dp_2$ on $\Delta$. 

The symmetry algebra $\g=\mathfrak{sl}(4,\RR)$ contains an element $\xi=u\p_u+p_1\p_{p_1}+p_2\p_{p_2}$. 
To obtain the corresponding symmetry reduction we pass to coordinates $(x^1,x^2,\ln(1/u),p_1/u,p_2/u)$,
where the first two and last two are invariants of the flow of $\xi$. Keeping the same notations $(x^1,x^2,u,p_1,p_2)$
for the new coordinates the quotient conformal metric on $\bar{M}=\RR^4(x^1,p_1,x^2,p_2)$ is given by the formula
 $$
\bar{g} = dx^1dp_1+dx^2dp_2+(p_1dx^1+p_2dx^2)^2.
 $$
This metric is self-dual but not conformally flat. 
A straightforward computation of conformal Killing vectors shows that this quotient corresponds to $SL(3)/GL(2)$.
(This can be also seen from Lie-theoretic arguments, since the normalizer of $\xi$ in $\g$ is $\mathfrak{gl}(3,\RR)$
with $\xi$ in the center.)

The corresponding CR-reduction (the case $\delta<0$) with $\g=\mathfrak{su}(2,2)$ gives $SU(1,2)/U(1,1)$.  

Note that for any $\xi\in\g$ the symmetry reduction allows to obtain not only 4D conformal structure, but also its canonical
metric representative. Indeed, for the symmetry $\xi$ (in the open dense set, where it is transversal to the contact distribution)  
the contact form can be normalized by the condition $\omega^0(\xi)=1$, whence we get a symplectic form $\Omega$ on 
$\Delta$ and normalization of the metric $\|\Omega^{-1}g\|=1$. This quotient metric $\bar{g}$ is almost 
(pseudo/para-)K\"ahler of neutral signature.

Actually, since in the flat model the torsion is zero, the metric $\bar{g}$ for $\delta<0$ is pseudo-Kähler 
and for $\delta>0$ is para-Kähler, i.e.\ the almost (para-)complex structure $J$ is integrable. 
The space of self-dual (pseudo/para-)K\"ahler structures (also known as Bochner-K\"ahler or 
Bochner-bi-Lagrangian structures, cf.\ \cite{CS09}) arising as the symmetry reduction of 
the flat 5D subconformal contact structure, depends on 3 parameters. 
More precisely, among the fundamental invariants \eqref{eq:tildetau}, one has $\tilde\tau_\pm=0,$ 
and, by \eqref{eq:Rest-Of-Schouten}, the self-dual Weyl curvature, given by the quartic \eqref{eq:Self-dual-Weyl-Curvature}, 
is zero if $r=0,$ or otherwise has Petrov type $D$ if $r\neq 0,$ i.e. has two repeated roots of multiplicity two.  
The unique (up to homothety) such non-conformally flat self-dual (para-)K\"ahler metric that is Einstein, 
i.e. satisfies $S=0$ and $R\neq 0$ in \eqref{eq:tildetau}, is the canonically defined pseudo-K\"ahler metric on 
$SU(1,2)/U(1,1)$ and para-K\"ahler metric on  $SL(3,\RR)/GL(2,\RR)$. 

\smallskip

{\bf 2.}
Now consider submaximally symmetric structures. For $A_3/P_{13}$ type geometry the submaximal
symmetry dimension is 8, achieved both in pure curvature and in pure torsion modules \cite{KT}. 
In the split real case $\delta>0$ the submaximal zero-curvature subconformal structue is unique 
and is given by the lift of the Egorov projective connection. 
We write it as a deformation of the trivial system of ODEs:
 \begin{equation}\label{Egorov}
\ddot{x}^1=0,\quad \ddot{x}^2=\epsilon\,x^1.
 \end{equation}
By the construction of Section \ref{Sec:4.1} this generates the zero-curvature structure in 5D, which as above, 
in coordinates $(x^1,x^2,u,p_1,p_2)$ is given by the $L_\pm$ splitting of the contact distribution:
 \begin{equation}\label{EgHV}
\Delta=H\oplus V\ \ \text{ with }\ \ 
H=\langle\p_{x^1}+p_1\p_u+\epsilon\,x^1\p_{p_2},\p_{x^2}+p_2\p_u+\epsilon\,x^1\p_{p_1}\rangle\ 
\text{ and }\ V=\langle\p_{p_1},\p_{p_2}\rangle.
 \end{equation}
The subconformal structure is represented by the metric $g=\omega^1\,\omega^3+\omega^2\,\omega^4$ 
in the coframe $\omega^1=dx^1$, $\omega^2=dx^2$, $\omega^3=dp_1-\epsilon\,x^1dx^2$, 
$\omega^4=dp_2-\epsilon\,x^1dx^1$ on $\Delta$. The point symmetry algebra of \eqref{Egorov} coincides
with the symmetry algebra $\g$ of the subconformal structure; it is solvable and generated by
 \begin{gather*}
\xi_1=x^1\p_{x^1}-x^2\p_{x^2}+u\p_u+2p_2\p_{p_2},\
\xi_2=x^1\p_{x^2}+\tfrac13\epsilon\,(x^1)^3\p_u+\bigl(\epsilon(x^1)^2-p_2\bigr)\p_{p_1},\\
\xi_3=x^2\p_{x^2}+u\p_u+p_1\p_{p_1},\
\xi_4=\p_{x^1}+\epsilon\,(x^1x^2\p_u+x^2\p_{p_1}+x^1\p_{p_2}),\\
\xi_5=\p_{x^2},\ \xi_6=\p_u,\
\xi_7=x^1\p_u+\p_{p_1},\ \xi_8=x^2\p_u+\p_{p_2}.
 \end{gather*}
Consider 
three Abelian subalgebras $\mathfrak{h}_1=\langle\xi_4,\xi_7\rangle$,
$\mathfrak{h}_2=\langle\xi_4,\xi_5\rangle$, $\mathfrak{h}_3=\langle\xi_4,\xi_3\rangle$.
All of them are non-null, meaning that the span of the fields does not intersect with the $g$-null cone on $\Delta$.

Let us first do symmetry reduction along $\xi_4$, which is the field common to all subalgebras. Passing to coordinates 
$\bigl(u-\frac12\epsilon\,(x^1)^2(x^2),x^2,x^1,p_1-\epsilon\,x^1x^2,p_2-\frac12\epsilon\,(x^1)^2\bigr)$
where the first two and last two are invariants of the flow of $\xi_4$ and keeping the same notations $(x^1,x^2,u,p_1,p_2)$
for the new coordinates, the quotient conformal metric on $\bar{M}=\RR^4(x^1,p_1,x^2,p_2)$ is given by the formula
 $$
\bar{g} = dx^1dp_1+dx^2(p_1dp_2-p_2dp_1)+\epsilon\,\tfrac{x^2}{p_1}(dx^1-p_2dx^2)^2.
 $$
This metric is self-dual but not conformally flat. 

We can take further reduction from 4D to 3D, in new coordinates the remaining generators of  
$\mathfrak{h}_1$, $\mathfrak{h}_2$, $\mathfrak{h}_3$ have the form: $\bar\xi_7=\p_{x^1}$,
$\bar\xi_5=x^2\p_{x^1}+\p_{p_2}$, $\bar\xi_3=x^1\p_{x^1}+x^2\p_{x^2}+p_1\p_{p_1}$.
Reduction along both $\bar\xi_7$ and $\bar\xi_5$ yields conformally flat 3D metrics, but the symmetry reduction 
along $\bar\xi_3$ gives in new coordinates/invariants $(x_1,x_2,x_3)$ the following conformal metric
 $$
\check{g}=(x_2x_3dx_1-(x_1x_3-1)dx_2+x_1x_2dx_3)^2+4\,(\epsilon\,x_1(x_1x_3-1)-x_2^2)\,dx_1dx_3.
 $$
This metric on the quotient $\check{M}=\RR^3(x_1,x_2,x_3)$ is not conformally flat, but is Einstein-Weyl.

\smallskip

{\bf 3.}
Let us describe the geometry on the twistor space $\mathcal{T}^4$ obtained from the correspondence space
$\hat{M}^6$ via quotient by the foliation of the dLp $\hat{\Pi}^2$. For the flat $(A_3,P_{13})$ geometry
the induced geometry on the twistor space is flat conformal, i.e.\ of type $(A_3,P_2)$. This follows from 
the parabolic twistor correspondence \cite{Cp}, see the discussion at the beginning of Section \ref{Sec:4}.
 
Otherwise the projection of tangents $\p_\lambda$ to the fibers of $\hat{M}^6\to M^5$ along the foliation $\hat{\Pi}^2$
yields a field of surfaces on $\cT$, i.e.\ a causal structure (see the discussion at the end of Section \ref{Sec:2.4});
thus the correspondence space can be identified with a codimension one sub-bundle of the projectivized tangent bundle: 
$\hat M\subset\mathbb{P}T(\mathcal{T})$. In general such a causal structure 
(also known as a cone structure) may not be \emph{isotrivial}, in the sense of \cite{Hwang}. 
Being isotrivial means that all the fibers of $\hat M\to\cT$  are projectively equivalent to a fixed projective surface 
$C\subset\PP^3$, and being isotrivially flat means trivialization of the bundle $\hat M^6\simeq C\times\cT$
as a subbundle of $\mathbb{P}T(\mathcal{T})$.

In the case of large symmetry, the causal structures are isotrivial. For the maximal symmetric case the causal structure
is isotrivially flat with quadratic fibers.
For the submaximal symmetric case it was demonstrated in \cite{Cayley} that $\hat M\to\cT$ is the so-called 
\emph{Cayley structure}, namely the fiber $C$ is the projectivized ruled Cayley cubic; moreover
it was proved that there are precisely two other 3D projective connections inducing a Cayley structure on $\cT$.
Let us show it directly for the zero-curvature subconformal structure, derived from the Egorov projective connection.

Denoting the vector generators in \eqref{EgHV} by $h_1,h_2$ for $H$ and $v_1,v_2$ for $V$, the dLp is
 $$
\hat{\Pi}=\langle h_1+\lambda h_2=\p_{x^1}+\lambda\p_{x^2}+(p_1+\lambda p_2)\p_u
+\epsilon\,x^1(\p_{p_2}+\lambda\p_{p_1}),\, v_2-\lambda v_1=\p_{p_2}-\lambda\p_{p_1}\rangle.
 $$ 
This distribution is Frobenius integrable on $\hat{M}^6$. Passing to new coordinates 
$\bigl( p_1+\lambda(p_2-\epsilon\,(x^1)^2), x^2-\lambda x^1, u + \lambda(\frac23\epsilon\,(x^1)^3-x^1p_2)-x^1p_1, 
x^1, p_2, \lambda\bigr)$, among which the first three and the last one are $\hat{\Pi}$-invariants, and
keeping old notations for the new coordinates, we get
 $$
\hat{\Pi}_\text{new}=\langle\p_{p_1},\p_{p_2}\rangle, \quad 
\langle\p_{\lambda}\rangle_\text{new}=
\big\langle\p_{\lambda}-(\epsilon\,p_1^2-p_2)\p_{x^1}-p_1\p_{x^2}+(\tfrac23\epsilon\,p_1^3-p_1p_2)\big\rangle.
 $$
Taking the coefficients of the latter vector field with simple rescaling 
$\big[1:p_1:p_2-\epsilon\frac{p_1^2}2,p_1p_2-\epsilon\frac{p_1^3}3\big]$ and changing coordinates once again
we get (isotrivially flat) ruled Cayley cubic
 $$
z=\tfrac13x^3-xy. 
 $$

{\bf 4.}
The submaximal symmetry algebra for subconformal structures in 5D with $\delta<0$ has dimension 7 \cite{KT,Mak2}.
There is no uniqueness in this case, and the models are obtained as follows. 
Consider a projective surface $C\subset\mathbb{P}^2$ with 2-dimensional symmetry algebra; 
for the classification see \cite{Doubrov} and references therein. Consider $\cT^4=\RR^4(x^1,x^2,x^3,x^4)$ and let
$\hat{M}^6=C\times\cT^4\subset\mathbb{P}T(\cT)$ be the isotrivially flat cone structure with the fiber $C$. 

By the construction, the correspondence space $\hat{M}^6$ has a rank 3 distribution split into two integrable
subdistributions $\hat{\Pi}^2\oplus\ell$. The projection along $\ell=\langle\p_\lambda\rangle$ yields $M^5$.
Surfaces $C$ with positive definite second fundamental form correspond to zero-curvature subconformal structures with 
$\delta<0$, while those with Lorentzian signature correspond to the para-CR case $\delta>0$.

The 7-dimensional symmetry is composed of four translations $\p_{x^i}$ and the homothety $x^i\p_{x^i}$ on 
$\cT$ as well as two projective symmetries of $C$. By functoriality these pass to $\hat{M}$ and $M$.
Explicit formulae can be obtained as follows.
Choose a coordinate system $(z^0,z^1,z^2,z^3)$ on $\cT$ and let $(w^0,w^1,w^2)$ be the corresponding 
affine coordinate chart for  $\PP^3=\PP T_z(\cT)$. An isotrivially flat causal structure $\hat M\subset \PP T(\cT)$, 
adapted to these coordinates, can be written as a graph
 \begin{equation*}
w^0=G(w^1,w^2).
 \end{equation*}
The induced (3,5,6) distribution on $\hat{M}$, following Section \ref{Sec:2.4}, is given as $\hat\Pi^2\oplus\ell$ with
 \[
\hat\Pi=\langle\partial_{w^1},\partial_{w^2}\rangle,\quad \ell=
\langle\textstyle G\frac{\partial}{\partial z^0}+w^1\frac{\partial}{\partial z^1}+w^2\frac{\partial}{\partial z^2}
+\frac{\partial}{\partial z^3}\rangle.
 \]
The first summand $\hat \Pi$ is the vertical tangent bundle to $\hat M\to\cT$;
the second summand $\ell$ is the characteristic line field for the odd contact 1-form 
 \[
\alpha=dz^0-G_{w^1}dz^1-G_{w^2}dz^2+\bigl(w^1G_{w^1}+w^2G_{w^2}-G\bigr)dz^3,
 \] 
which at $(z;w)\in\hat{M}$ belongs to the pullback of the annihilator to the affine tangent space of the cone 
$\hat C_z\subset T_z\cT$ along $w\in C_z$. 

\smallskip

{\bf 5.}
Finally we consider the twistorial construction of 2-nondegenerate (para-)CR structure in dimension 7. It is a bundle 
$\hat{\cT}_\varepsilon$ over $M^5$ with the two-dimensional fiber $\mathbb{D}^2$ or $\mathbb{S}^1\times\RR^1$
according to Lemma \ref{lmT}. Denoting the coordinates in the fiber by $q_1,q_2$ and keeping
the notations $h_1,h_2$ for the generators of $H$ and $v_1,v_2$ for the generators of $V$ in \eqref{EgHV} 
the induced para-CR structure ($\varepsilon=+1$) corresponding to the Egorov structure \eqref{Egorov} is
given by the splitting 
 $$
D= \langle h_1+q_1h_2,v_2-q_1v_1,\p_{q_2}\rangle \oplus \langle h_2+q_2 h_1,v_1-q_2v_2,\p_{q_2}\rangle
 $$
into a pair of integrable subdistributions. Similarly, one gets a CR structure ($\varepsilon=-1$) corresponding 
to the Egorov structure via a pair of complex conjugated subdistributions in $D\otimes\CC$.

This twistorial construction is fully functorial: an equivalence downstairs lifts from $M^5$  to $\hat{\cT}_\varepsilon$ and, 
conversely, any symmetry of $\hat{\cT}_\varepsilon$ projects along the Levi kernel $K$ to $M^5$. 
The distribution $D$ projects to the contact structure $\Delta$. The (para-)CR structure $\mathcal{J}_\varepsilon$ 
is not projectable, but there exists a unique up to sign (para-)complex structure $\hat{J}$, 
commuting with it, cf.\ the proof of Theorem \ref{thm:2-nondegenerate-cr-2}. 
(This is analogsous to the left and right (split-)quaternionic multiplications.) 
Such $\hat{J}$ projects along $K$ to $J$ on $M^5$ making it into a zero-curvature subconformal manifold.
 
We note that 7-dimensional 2-nondegenerate CR structures arising from zero-curvature subconformal structures 
are \emph{recoverable} in the sense of \cite{SZ}. A CR structure is called recoverable if the space
\[\mathrm{ad}(K''_x):=\{\mathrm{ad}_v\,|\, v\in K''_p\}\subset \mathrm{Hom}(D'_x\slash K'_x,D''_x\slash K''_x)\]
 has vanishing first prolongation. Recall that the first prolongation is defined as the kernel of the Spencer operator $\partial\colon \mathrm{Hom}(V,W)\to \mathrm{Hom}(V\w V,W)$ where  $\partial f(v,u)=f(v)u-f(u)v.$ Using the relation \eqref{eq:2nondegeneracy}, it is straightforward to show that the first prolongation of $\mathrm{ad}(K''_x)$ 
 is zero for all $x\in\hat\cT_\ve$. This gives isomorphism of symmetry algebras. 

In particular, for the family derived from \eqref{Egorov} with $\epsilon\neq0$ we obtain 7-dimensional 2-nondegenerate 
CR structure with 8D symmetry, while for the flat case $\epsilon=0$ we get the structure with 15D symmetry algebra. 
This latter can be either $SL(4,\RR)$ or $SU(2,2)$ realizing the submaximal
symmetry dimension for 7-dimensional 2-nondegenerate CR structures.

\section{Outlook}\label{Sec:5}

In this paper we  generalized the paradigm ``Integrability via Geometry'' \cite{CK} from 3D and 4D
to dispersionless equations and Lax pairs in 5D, provided that the symbol of the equation has 
conformal symmetric bivector of rank 4 and the corresponding distribution is contact at generic solution 
$u$ of the equation $\E$.
Surprisingly this case has an underlying parabolic geometry, similar to the well-studied lower-dimensional cases.
This allowed us to identify the required curvature component of the corresponding subconformal structure.

Let us note that parabolic geometries of type $(A_3,P_{13})$ have been studied in the literature, 
for instance subconformal structures of neutral signature in 5D appeared 
both as integrable Legendrian and Levi-split CR structures. However in that class 
the torsion vanishes while the curvature may be non-zero. In our case the situation is opposite: 
the dispersionless integrability is equivalent to the vanishing of the curvature, while we allow torsion.

Some of the concepts we introduced in this paper have direct higher dimensional generalizations.
The fact that the underlying geometry is parabolic does not persist in higher dimensions,  
as we indicate in the Appendix. In our next paper 
we will address those and explain how to define the proper curvature and connect it to the integrability.

If the rank of $c_\E$ drops below the maximal value 4 or the distribution becomes (partially) integrable
for non-generic $u$, this yields a class of algebraically special solutions.
If however this happens for generic solution $u$ of the PDE $\E$,  then the underlying geometry is different.

For instance, investigating rank 3 subconformal structures on a 5D background generically meets the following aspects: 
(i) the distribution $\Delta$ has growth vector $(3,5)$ and the radical $\sqrt{\Delta}$, a rank 2 distribution with 
growth vector $(2,3,5)$ on a generic solution $u$, has higher order tensorial invariants;
(ii) a partial Weyl connection associated to the subconformal structure $c_\E$ has curvature components on its own;
(iii) a compatibility condition relates both (i) and (ii).

We expect that integrability for such a class of 5-dimensional dispersionless PDEs is also expressible 
via  certain ``zero-curvature condition'', involving the above invariants.  We do not address these question here,
but raise the problem of describing the corresponding integrable background geometry as a parabolic geometry. 
We also expect that 5D zero-curvature systems discussed in this paper pass the test for hydrodynamic integrability 
as in \cite{FK} as well as adaptation for the deformation scheme as in \cite{St}.

\appendix
\setcounter{equation}{0}
\setcounter{subsection}{0}
 \setcounter{theorem}{0} 
 % ...

\section*{Appendix: Integrable parabolic background geometries}\label{ApB}
\renewcommand{\theequation}{A.\arabic{equation}}
\renewcommand{\thesection}{A}

Let us define a class of parabolic geometries that may serve as \emph{integrable background geometry}
by first specifying the compatibility properties discussed in Section \ref{Sec:1.1}, which only constrain 
the algebraic type of the geometry and does not restricts the curvature properties.

\begin{definition}\label{Df1}
Given a semi-simple Lie group $G$ and a parabolic subgroup $P\subset G$, a parabolic geometry of type $(G,P)$ 
is called a \emph{compatible background geometry} if it admits a subconformal structure $(\g_{-1},c)$
with the following compatibility conditions:
 \begin{itemize}
\item $\dim\g_{-1}=3\text{ or }4$, 
\item there exists a (unique)  $G_0$-invariant conformal structure $c$ on $\g_{-1}$,
\item there exists a 1-parameter family of $c$-null 2-planes that are Abelian in $\m=\g_{-}$.
 \end{itemize}
These 2-planes are henceforth called $\a$-planes. 
\end{definition}

\begin{definition}
A compatible parabolic background geometry is called an \emph{integrable background geometry} 
if every $\a$-plane is tangent to an $\a$-surface, i.e.\ a surface whose tangents are $\a$-planes. 
\end{definition}

Note that we included the case of rank 3 subconformal structure, whose integrability requires a choice  of Weyl connection,
which is an additional constraint on the parabolic geometry.
Geometries of rank  2 and 1 may also be considered as an instance of integrable background geometry,
within the paradigm of \cite{C}, however in this paper we restrict to ranks 3 and 4.

The uniqueness claim in Definition \ref{Df1} is not obvious a priori but is a by-product of the following classification result.
 \begin{theorem}\label{IBG}
The only compatible parabolic background geometries are those of the types $(B_2,P_1)$, 
$(D_3,P_1)$ and $(D_3,P_{2,3})$, as well as $(B_3,P_{13})$ and $(C_3,P_2)$. 
 \end{theorem}

Here we used the exceptional isomorphism $(A_3,P_{1,3})=(D_3,P_{2,3})$ over $\CC$.
Over $\RR$ the compatible parabolic background geometries of the above type are:
$(\mathfrak{so}(2,3),\fp_1)$, $(\mathfrak{so}(3,3),\fp_1)$, $(\mathfrak{so}(3,3),\fp_{23})$,
$(\mathfrak{su}(2,2),\fp_{13})$, $(\mathfrak{so}(3,4),\fp_{13})$, $(\mathfrak{sp}(6,\RR),\fp_2)$.

 \begin{proof}
We will work with Lie algebras. The constraint of $\g_0$-equivariance is enough to exclude false candidates.
For the remaining five, the equivariance on the group level is straightforward. Untill the end of the
proof we work over $\CC$ (the claim over $\RR$ follows by direct inspection).

We start with simple Lie algebras. Their structure root theory implies $\dim\g_{-1}\geq n=\op{rank}(\g)$ 
(the equality is achieved on the Borel parabolic subalgebra $\fp_{1\dots n}$). Indeed,
if $\a_i$ are simple roots, then $\sum_{i\in S}\a_i$ is also a root for any connected piece $S$ of the Dynkin diagram.
Thus, there are at least $n$ roots of this form, where $S$ contains only one crossed node in $\Sigma$
for $\fp=\fp_\Sigma$.

This already makes the list of candidates finite. Further restrictions are:
 \begin{itemize}
\item Parabolics given by more than 2 crosses ($\fp_\Sigma$ with $|\Sigma|>2$) have
3 independent scalings, and so cannot conformally preserve a bi-linear form of rank $\ge3$;
\item The Dynkin diagram of the Levi part $\g_0^{ss}$, obtained by excluding crosses,
must be $B_1=A_1$ or $D_2=A_1A_1$, because in other cases $\g_0^{ss}$ cannot preserve a conformal structure.
 \end{itemize}
Consequently, for $\dim\g_{-1}=3$ the candidates for compatible parabolic background geometries are 
(we use outer automorphisms of $\g$ to exclude repetitions):
 \begin{gather*}
(A_3,P_{12}),\ (B_2,P_1),\ (B_3,P_{23}),\ 
(C_3,P_{12}),\ (C_3,P_{23}).
 \end{gather*}
 For $\dim\g_{-1}=4$ the candidates are (we use exceptional isomorphisms $B_2=C_2$, $A_3=D_3$):
 \begin{gather*}
(A_3,P_2),\ (A_3,P_{13}),\ (A_4,P_{23}),\ 
(B_3,P_{12}),\ (B_3,P_{13}),\
(C_3,P_2),\ (C_4,P_{23}),\ 
(G_2,P_2),\ (F_4,P_{23}).
 \end{gather*}

Among these only the following have $\g_0$-invariant conformal structure
 \begin{equation}\label{list}
(B_2,P_1);\
(A_3,P_2),\ (A_3,P_{13}),\ (B_3,P_{13}),\ (C_3,P_2).
 \end{equation}
The recipe to do so is as follows: determine the $\g_0$-module $\g_{-1}$ and compute its $[2]$ plethysm; for instance
in the case $(B_3,P_{23})$ we have $\g_0^{ss}=A_1$ and $\g_{-1}=1\!\times\![0]+1\!\times\![1]$,
whence $S^2\g_{-1}=1\!\times\![0]+1\!\times\![1]+1\!\times\![2]$ but the trivial representation 
corresponds to rank 1 bilinear form, and hence an invariant conformal metric does not exist.

The first case in \eqref{list} has 3-dimensional irreducible manifold ($\g_{-1}=TM$), and the remaining cases have
4-dimensional $\g_{-1}$: on the manifolds of dimensions 4, 5, 8 and 7, respectively.
A straigtforward computation shows that each is a compatible background geometry.

If $\g$ is semi-simple but not simple, it must be a product of
 \begin{equation}\label{list2}
(A_1,P_1),\ (A_2,P_1),\ (A_2,P_{12}),\ (B_2,P_2),\ (B_2,P_{12}),\ (G_2,P_1),\ (G_2,P_{12})
 \end{equation}
with $\dim\g_{-1}=1\vee 2$. In this list only geometries with the Borel subgroup $B\subset G$ possess
$\g_0$-invariant conformal structure on $\g_{-1}$. However in the product, i.e.\ for non-simple $G$,
with factors from \eqref{list} and \eqref{list2}, no such geometry can have $\g_0$-invariant conformal structure on $\g_{-1}$
(there are however $\g_0^{ss}$-invariant such). 
This finishes the proof.
 \end{proof}
 
 \begin{corollary}
The only integrable parabolic background geometries are those listed in Theorem \ref{IBG} with their respective 
corresponding zero-curvature condition. 
  \end{corollary}

These so-called ``zero-curvature'' conditions were already discussed for the standard conformal geometries  $(B_2,P_1)$ in 3D, 
with a choice of Weyl structure, and $(D_3,P_1)$ in 4D,   as well as for the geometry $(A_3,P_{13})$ studied in this paper.
The novel candidates are $(B_3,P_{13})$ and $(C_3,P_2)$ in 8D and 7D, respectively.
While higher dimensions will be considered in a forthcoming paper (where the corresponding curvature $W$ will be introduced)
let us briefly comment on the two new cases.

\smallbreak

{\bf 8D case.} The grading of $B_3$ corresponding to parabolic $P_{13}$ is
$\g=\g_{-3}\oplus\g_{-2}\oplus\g_{-1}\oplus\g_0\oplus\g_1\oplus\g_2\oplus\g_3$
with dimensions of components $(2,2,4,5,4,2,2)$. In terms of root vectors,
$\g_0$ is generated by the Cartan subalgebra $\h$ and $e_{\pm\a_2}$, while
 $$
\g_{-1}=\langle e_{-\a_1},e_{-\a_1-\a_2},e_{-\a_3},e_{-\a_2-\a_3}\rangle\
\text{ and }\ \g_{-2}=\langle e_{-\a_1-\a_2-\a_3},e_{-\a_2-2\a_3}\rangle.
 $$
The only $\g_0$-invariant conformal metric (after proper normalization of root vectors) is
 $$
g=e_{\a_1}\cdot e_{\a_2+\a_3}+e_{\a_1+\a_2}\cdot e_{\a_3} 
 $$
The only congruence of $\a$-planes ($g$-null and $[,]$-isotropic) is
 $$
\langle e_{-\a_1}+\l e_{-\a_1-\a_2}, e_{-\a_3}-\l e_{-\a_2-\a_3}\rangle.
 $$
This lifts to a canonical rank 2 distribution in the corresspondence space, which is a parabolic geometry of
type $(B_3,P_{123})$ and a bundle over $(B_3,P_{13})$ with fiber $\PP^1$. However
the curvature $W$, corresponding to the Frobenius condition, takes value in the component of 
the cohomology $H^2(\g_-,\g)$, which is a $\g_0^{ss}$ irreducible module with the lowest weight vector
 $$
e_{\a_1}\wedge e_{\a_3}\otimes e_{-\a_1-2\a_2-2\a_3}.
 $$
This corresponds to torsion of negative homogeneity $-1$ with respect to the grading element $Z\in\g_0$.
We refer for the technique behind this computation and the general theory of parabolic geometries to \cite{CS}.
In non-trivial case $W\neq0$, this means that the parabolic geometry is non-regular and hence the zero-curvature condition 
cannot be obtained as a component of the harmonic curvature in the standard parabolic formalism.

\smallbreak

{\bf 7D case.} The grading of $C_3$ corresponding to parabolic $P_2$ is
$\g=\g_{-2}\oplus\g_{-1}\oplus\g_0\oplus\g_1\oplus\g_2$
with dimensions of components $(3,4,7,4,3)$. In terms of root vectors,
$\g_0$ is generated by the Cartan subalgebra $\h$ and $e_{\pm\a_1},e_{\pm\a_3}$, while
 $$
\g_{-1}=\langle e_{-\a_2},e_{-\a_1-\a_2},e_{-\a_2-\a_3},e_{-\a_1-\a_2-\a_3}\rangle
\text{ and } \g_{-2}=\langle e_{-2\a_2-\a_3},e_{-\a_1-2\a_2-\a_3},e_{-2\a_1-2\a_2-\a_3}\rangle.
 $$
The only $\g_0$-invariant conformal structure, after proper normalization of root vectors, is
 $$
g=e_{\a_2}\cdot e_{\a_1+\a_2+\a_3}+e_{\a_1+\a_2}\cdot e_{\a_2+\a_3} 
 $$
The only congruence of $\a$-planes ($g$-null and $[,]$-isotropic) is
 $$
\langle e_{-\a_2}+\l e_{-\a_2-\a_3}, e_{-\a_1-\a_2}-\l e_{-\a_1-\a_2-\a_3}\rangle.
 $$
This lifts to a canonical rank 2 distribution in the corresspondence space, which is a parabolic geometry of
type $(C_3,P_{23})$ and a bundle over $(C_3,P_2)$ with fiber $\PP^1$. 
In fact, the preimage of the above congruence is the distribution with growth  $(3,2,3)$ in the correspondence space.
The curvature $W$, corresponding to the Frobenius condition, takes value in the component of 
the cohomology $H^2(\g_-,\g)$, which is a $\g_0^{ss}$ irreducible module with the lowest weight vector
 $$
e_{\a_2}\wedge e_{\a_1+\a_2}\otimes e_{-\a_3}.
 $$
This has homogeneuity $2$ with respect to the grading element $Z\in\g_0$, so the geometry is regular.
Thus, the corresponding  $W$ component of $\kappa_H^2$ can be computed through the technique
of parabolic geometry, similar as we did with the 5D subconformal case in this paper.

The zero-curvature condition $W=0$ of the subconformal geometry is however not sufficient to identify the twistor 
space with the parabolic geometry of type $(C_3,P_3)$ unless the initial parabolic geometry of type $(C_3,P_2)$ is flat. Indeed, there is another harmonic curvature component $\kappa_H^1$ of homogeneity 1, 
which must vanish in order for parabolic geometry to descent. 
This is yet another analog of the classical conformal geometry in 4D, in which case the twistor space is never
projective unless the conformal structure is flat.

\bigskip

\begin{center}{\sc Funding}\end{center}

\smallskip

The research leading to our results has received funding from the Norwegian Financial Mechanism 2014-2021
(project registration number 2019/34/H/ST1/00636), the Polish
National Science Centre (NCN) (grant number 2018/29/B/ST1/02583), and the Tromsø Research Foundation
(project “Pure Mathematics in Norway”). OM gratefully acknowledges partial support by the grant  PID2020-116126GB-I00 provided via the Spanish Ministerio de Ciencia e Innovaci\'on MCIN/ AEI /10.13039/50110001103.
 

\end{document}